\documentclass[11pt]{article}
\usepackage{amsfonts,bm}
\usepackage{amssymb}
\usepackage{latexsym}
\usepackage{mathrsfs}
\usepackage{multirow}
\usepackage{multicol}
\usepackage{graphicx}
\usepackage{rotating}
\usepackage{subfigure}
\usepackage{color}
\usepackage{algorithm}
\usepackage{algorithmic}
\usepackage{setspace}
\usepackage{cite}
\usepackage{epstopdf}
\usepackage{amsthm}
\usepackage[cmex10]{amsmath}
\usepackage{amsmath}
\usepackage{bm}
\usepackage{url}
\usepackage{hyperref}
\usepackage{indentfirst}
\usepackage{booktabs}
\usepackage{float}
\usepackage{authblk}
\usepackage[pagewise]{lineno}
\newtheorem{theorem}{Theorem}[section]
\newtheorem{proposition}[theorem]{Proposition}
\newtheorem{lemma}{Lemma}[section]
\newtheorem{example}[theorem]{Example}

\numberwithin{equation}{section}

\date{}
\begin{document}
\title{Fast second-order implicit difference schemes for time distributed-order and Riesz space fractional diffusion-wave equations}

\author[1]{Huan-Yan Jian\thanks{\textit{E-mail address:} uestc\_hyjian@sina.com}}
\author[1]{Ting-Zhu Huang \thanks{Corresponding author. \textit{E-mail address:} tingzhuhuang@126.com. Tel.: 86-28-61831016}}
\author[2]{Xian-Ming Gu \thanks{Corresponding author. \textit{E-mail address:} guxianming@live.cn, x.m.gu@rug.nl}}
\author[1]{Xi-Le Zhao \thanks{\textit{E-mail address:} xlzhao122003@163.com}}
\author[1]{Yong-Liang Zhao \thanks{\textit{E-mail address:} ylzhaofde@sina.com}}
\affil[1]{\footnotesize School of Mathematical Sciences, University of Electronic
Science and Technology of China, Chengdu, Sichuan 611731, P.R. China.}
\affil[2]{\footnotesize School of Economic Mathematics/Institute of Mathematics,
Southwestern University of Finance and Economics, Chengdu, Sichuan 611130, P.R. China.}
%latex文件输出pdf文字大小设置:\tiny\scriptsize\footnotesize\small\normalsize\large\Large\LARGE\huge\Huge
\maketitle
\begin{center}
\textbf{Abstract}
\end{center}
In this paper, fast numerical methods are established for solving a class of time distributed-order and Riesz space fractional diffusion-wave equations. We derive new difference schemes by the weighted and shifted Gr$\ddot{\rm{u}}$nwald formula in time and the fractional centered difference formula in space. The unconditional stability and second-order convergence in time, space and distributed-order of the difference schemes are analyzed.
In the one-dimensional case, the Gohberg-Semencul formula utilizing the preconditioned Krylov subspace method is developed to solve the symmetric positive definite Toeplitz linear systems derived from the proposed difference scheme.
In the two-dimensional case, we also design a global preconditioned conjugate gradient method with a  truncated preconditioner to solve the discretized Sylvester matrix equations.
We prove that the spectrums of the preconditioned matrices in both cases are clustered around one, such that the proposed numerical methods with preconditioners converge very quickly.
%the Krylov subspace method with the proposed preconditioner converges superlinearly.
Some numerical experiments are carried out to demonstrate the effectiveness of the proposed difference schemes and show that the performances of the proposed fast solution algorithms are better than other numerical methods.


\section{Introduction}\label{section1}

In recent years, fractional diffusion equations (FDEs) have gained more and more attention since their widespread applications in modeling complex processes such as finance \cite{Korbel2016Modeling}, biology \cite{Ionescu2017The}, chaos \cite{Baleanu2015Chaos}, optimal control \cite{Bhrawy2015An}, signal processing\cite{Li2018Time} and random walk \cite{Metzler2000The}. In general, these models are constructed in the form of single or multi-term time, space or time-space FDEs. However, when the processes lack temporal scaling, neither single nor multi-term FDEs can describe them. Therefore distributed-order FDEs are introduced to model the processes that become more anomalous in course of time, e.g. the ultraslow diffusion or accelerating superdiffusion where a plume of particles spreads at a logarithmic rate \cite{Kochubei2007Distributed,Chechkin2002Retarding}.

The idea of distributed-order FDEs was first introduced by Caputo \cite{Caputo1995Mean} to generalize the stress-strain relations of unelastic media.
In \cite{Gorenflo2013Fundamental}, Gorenflo et al. developed the fundamental solution for a 1D distributed-order fractional diffusion-wave equation by using the Fourier-Laplace transform and the interpolation technique.
Li et al. \cite{Li2016Analyticity} discussed an initial-boundary value problem of a time distributed-order FDE with the Caputo fractional derivative, where the analytical solution in time was obtained.
 %and the uniqueness and existence of the solution were proved by eigenfunctions expansion and Laplace transform.
However, the analytical solutions of many distributed-order FDEs are not easy to gain and the uniqueness and existence of the analytical solutions are not easy to prove. Therefore, different numerical methods for solving the distributed-order FDEs are considered.

Ye et al. \cite{Ye2015Compact} derived a compact difference scheme of a distributed-order fractional diffusion-wave system and demonstrated the unconditional stability and convergence of the difference scheme.
A new numerical method based upon hybrid functions approximation was proposed by Mashayekhi et al. \cite{Mashayekhi2016Numerical} to solve the distributed-order FDEs, which consists of Bernoulli polynomials and block-pulse functions.
Bu et al. \cite{bu2017finite} developed a numerical scheme to solve distributed-order time FDEs with the finite element method in the space direction and the L1 method in the time direction. Then also gave the analysis of unconditional convergence and stability of the numerical method.
In \cite{Li2016Galerkin}, Li et al. studied Galerkin finite element methods with non-uniform temporal meshes for multi-term Caputo-type FDEs, and proved that the corresponding finite element schemes were unconditionally convergent and stable.
Gao et al. \cite{Gao2017Two} constructed two difference schemes to solve 1D time distributed-order fractional wave equations and derived them from a weighted and shifted Gr$\ddot{\rm{u}}$nwald formula with second-order accuracy.

However, the current research is still sparse about the numerical methods of distributed-order FDEs. This motivates us to develop efficient numerical solutions for the following time distributed-order and Riesz space fractional diffusion-wave equations:
\begin{align}
&\int_1^2\omega(\alpha){}^C_0D_t^{\alpha}u(x,t)d\alpha=K\frac{\partial^{\beta}u(x,t)}{\partial\mid
x\mid^{\beta}}+f(x,t),\quad 0<x<L,~0<t\leq T,&\label{1.1}\\
&u(x,0)=0,\quad u_t(x,0)=0,\quad 0<x<L,\label{1.2}&\\
&u(0,t)=0,\quad u(L,t)=0,\quad 0\leq t\leq T,&\label{1.3}
\end{align}
where $\alpha\in[1,2],~\beta\in(1,2]$, $K>0$, and $f(x,t)$ is the source term.
%$\phi_1(0)=(\phi_1)_t(0)=\phi_2(0)=(\phi_2)_t(0)=0$.
Here, $^C_0D_t^{\alpha}u(x,t)$ is the Caputo fractional derivative \cite{Ye2015Compact} defined as:
\begin{equation*}
^C_0D_t^{\alpha}u(x,t)=
\begin{cases}
u_t(x,t)-u_t(x,0),\quad \alpha=1,\\
\frac{1}{\Gamma(2-\alpha)}\int_0^t(t-\xi)^{1-\alpha}\frac{\partial^2{u}}{\partial\xi^2}(x,\xi)d\xi,\quad 1<\alpha<2,\\
u_{tt}(x,t), \quad \alpha=2,
\end{cases}
\end{equation*}
and the weight function $\omega(\alpha)$ satisfies the conditions
\begin{equation*}
0\leq\omega(\alpha),~\omega(\alpha)\neq0,~0<\int_1^{2}\omega(\alpha)d\alpha<\infty,~~\alpha\in[1,2].
\end{equation*}
Moreover, the right-side $\frac{\partial^{\beta}}{\partial\mid x\mid^{\beta}}u(x,t)$ is the Riesz fractional derivative of order $\beta\in(1,2]$, whose definition is given by \cite{jiang2012analytical} %\cite{jiang2012analytical}
\begin{equation*}
\frac{\partial^{\beta}u(x,t)}{\partial\mid x\mid^{\beta}}=
\begin{cases}
-\frac{1}{2\cos(\beta\pi/2)\Gamma(2-\beta)}\frac{d^2}{dx^2}\int_0^L\mid x-\xi\mid^{1-\beta}u(\xi,t)d\xi,\quad 1<\beta<2,\\
u_{xx}(x,t),\quad \beta=2,
\end{cases}
\end{equation*}
with $\Gamma(\cdot)$ denoting the Gamma function.

To establish the high-order accurate numerical method to solve the problem \eqref{1.1}-\eqref{1.3}, we first transform Eqs. \eqref{1.1}-\eqref{1.3} into multi-term time-space diffusion-wave equations via the composite trapezoid formula \cite{gao2015some}. Then by applying the weighted and shifted Gr$\ddot{\rm{u}}$nwald formula \cite{Gao2017Two} to discrete the time-fractional derivatives of the resulted multi-term equations, the second-order accurate approximation in time can be achieved. In addition, to gain the second-order accuracy in space, the fractional centered difference formula \cite{ye2014numerical} is used to approximate the Riesz derivative. Hence, a new second-order difference scheme in all variables is developed.

When the above discrete methods are used to approximate the problem \eqref{1.1}-\eqref{1.3}, the real symmetric positive definite (SPD) Toeplitz linear systems will be obtained. Because of the nonlocal characteristics \cite{Pan2017Fast} of the fractional derivative, the coefficient matrices are usually dense or even full. If the traditional methods like Cholesky factorization are used to solve the linear systems, the costs are significantly expensive, which require $\mathcal{O}(M^2)$ storage and $\mathcal{O}(M^3)$ computational work, where $M$ is the number of grid points. Since the coefficient matrix contains the Toeplitz structure, Krylov subspace methods (KSMs) are considered to solve the linear systems. Because KSMs only need the information of matrix-vector product that can be done in only $\mathcal{O}(M\log M)$ operations via the fast Fourier transform (FFT) and the storage requirement is only $\mathcal{O}(M)$ \cite{Ng2004Iterative,Pan2014Preconditioning}. Various KSMs for the Toeplitz linear systems have been developed and studied \cite{chan1989toeplitz,gu2015k,gu2015hybridized}.

%It is well known that for a Toeplitz matrix, the matrix-vector product can be done in only $\mathcal{O}(M\log M)$ operations via the fast Fourier transform (FFT) and the storage requirement is only $\mathcal{O}(M)$ \cite{Ng2004Iterative,Pan2014Preconditioning}.
Since the coefficient matrix has a SPD Toeplitz structure, its inverse matrix can be explicitly expressed by the Gohberg-Semencul formula (GSF)\cite{Pang2011Shift}. On the other hand, the coefficient matrix is also time-independent. Thus its inverse can be calculated only once and the cost for computing the inverse matrix-vector multiplication only needs four FFTs. In this work, the conjugate gradient (CG) method \cite{chan1989toeplitz} is employed to solve the SPD Toeplitz linear system arising from the GSF.
However, the CG method converges slowly when the coefficient matrix is ill-conditioned. To accelerate its convergence, some preconditioners \cite{lei2013circulant,leichenzhang2016,Zhao2018Afast,Donatelli2016Spectral} are designed according to the structures of the the coefficient matrices.

In the 2D case, the GSF cannot be used since the coefficient matrix is a block Toeplitz with Toeplitz blocks (BTTB) matrix and not a Toeplitz matrix. Therefore, a new method needs to be designed to solve the resulting Sylvester matrix equations. In \cite{Jafarbigloo2013Global}, S. Karimi presented a global CG (GL-CG) method for solving a large Sylvester matrix equation, and demonstrated that it was more effective than the CG method to deal with the corresponding Kronecker equation. In order to improve the convergence, we propose a global preconditioned CG (GL-PCG) method with a preconditioner to solve the Sylvester matrix equations. Moreover, it is well-known that a block circulant with circulant blocks (BCCB) \cite{Capizzano1999Any} preconditioner of a BTTB matrix cannot be optimal, and a BTTB preconditioner of a BTTB matrix is optimal even if the eigenvalues of the preconditioner do not closely approximate the eigenvalues of the coefficient matrix \cite{Serra1994Preconditioning}. Therefore, in this paper, we design a truncated preconditioner which has the same structure as the coefficient matrix to solve the 2D problem.

This paper has two goals: (1) to propose unconditionally stable difference schemes with second-order accuracy in time, space and distributed-order to solve the 1D and 2D problems;
(2) to explore efficient algorithms for solving the resulted linear systems.
In the 1D case, based on the SPD Toeplitz structure, the GSF utilizing the preconditioned KSM  (referred to as PKSM-based GSF) is developed to process the linear systems, whose computational cost per iteration is only $\mathcal{O}(M\log M)$.
Since the PKSM-based GSF cannot be directly applied to solve the 2D problem, we design the GL-PCG method with a truncated preconditioner to solve the discretized Sylvester matrix equations.
It will be demonstrated that the spectrums of the preconditioned matrices are clustered around one. Therefore, the proposed numerical methods with preconditioners converge very quickly.

The outline of this paper is as follows. In Section \ref{section2}, we introduce the derivation of the new difference scheme for the problem $\eqref{1.1}$-$\eqref{1.3}$. The unique solvability, unconditional stability, and second-order convergence of the difference scheme are proved in Section \ref{section3}. In Section \ref{section4}, the PKSM-based GSF is presented to solve the resulted SPD Toeplitz linear systems. The spectral properties of the preconditioned matrix are discussed in Section \ref{section5}. In Section \ref{section6}, the GL-PCG method with a truncated preconditioner is designed to handle the 2D problem. In Section \ref{section7}, numerical experiments are provided to verify the second-order convergence of the difference schemes and show the efficiency of the proposed fast solution techniques. Finally, concluding remarks are given in Section \ref{section8}.

\section{The derivation of the difference scheme}\label{section2}
In this section, we focus on deriving the new difference scheme of the problem $\eqref{1.1}$-$\eqref{1.3}$.

For $M,N\in\mathbb{N}^{+}$, let $h=\frac{L}{M}$ and $\tau=\frac{T}{N}$ be the spatial grid size and time step, respectively. Then we can define $ x_i=ih~(0\leq i\leq M)$ and $t_n=n\tau~(0\leq n\leq N)$. The domain $[0,L]\times[0,T]$ is covered by $\Omega_h\times\Omega_\tau$, where $\Omega_h=\{x_i \mid x_i=ih,~0\leq i\leq M\}$ and $\Omega_\tau=\{t_n \mid t_n=n\tau,~0\leq n\leq N\}$. Let $u=\{u_i^n\mid~0\leq i\leq M,~0\leq n\leq N\}$ be a grid function on $\Omega_h\times\Omega_\tau$ and define
$$u_i^{n-\frac{1}{2}}=\frac{1}{2}(u_i^n+u_i^{n-1}),~~\delta_tu_i^{n-\frac{1}{2}}=\frac{1}{\tau}(u_i^n-u_i^{n-1}).$$

Then the following lemmas are needed to derive the difference scheme.
% Lemma 2.1
\begin{lemma}(\hspace*{-0.35em} The composite trapezoid formula \cite{gao2015some}) %\cite{gao2015two,gao2015some})
Let $z(\alpha)\in C^2([1,2])$, then we have
\begin{equation*}
\int_1^2z(\alpha)d\alpha=\Delta\alpha\sum\limits_{l=0}^{2J} c_lz(\alpha_l)-\frac{\Delta\alpha^2}{12}z^{(2)}(\eta),\quad \eta\in(1,2),
\end{equation*}
where $\Delta\alpha=\frac{1}{2J}$, $\alpha_l=1+l\Delta\alpha~(0\leq l\leq2J)$ and
\begin{equation*}
c_l=
\begin{cases}
\frac{1}{2},\quad l=0,~2J,\\
1,\quad 1\leq l\leq2J-1.
\end{cases}
\end{equation*}
\label{lemma2.1}
\end{lemma}

%Lemma 2.2
\begin{lemma} \cite{ye2014numerical}  %\cite{ye2014numerical}
Suppose that $u(x)\in C^5[0,L]$ satisfy the boundary condition $u(0)=u(L)=0$. The fractional centered difference formula for approximating the Riesz derivatives when $1<\beta\leq2$ is as follows:
\begin{equation*}
\frac{\partial^{\beta}u(x_i)}{\partial| x|^{\beta}}=
-h^{-\beta}\sum\limits_{k=i-M}^i\hat{g}_k^{(\beta)}u(x_{i-k})+\mathcal{O}(h^2),
\end{equation*}
where
\begin{equation*}
\hat{g}_k^{(\beta)}=\frac{(-1)^k\Gamma(\beta+1)}{\Gamma(\beta/2-k+1)\Gamma(\beta/2+k+1)}.
\end{equation*}
\label{lemma2.2}
\end{lemma}

Consider $\eqref{1.1}$ at point $(x_i,t_n)$:
\begin{equation}\label{2.1}
\int_1^2\omega(\alpha){}^C_0D_t^{\alpha}u(x_i,t_n)d\alpha=K\frac{\partial^{\beta}u(x_i,t_n)}{\partial\mid
x\mid^{\beta}}+f(x_i,t_n),\quad 1\leq i\leq M-1,~0\leq n\leq N.
\end{equation}
Take an average of $\eqref{2.1}$ on time levels $t=t_{n-1}$ and $t=t_n$ to get
\begin{equation}
\begin{split}
&\frac{1}{2}\left(\int_1^2\omega(\alpha){}^C_0D_t^{\alpha}u(x_i,t_n)d\alpha+
\int_1^2\omega(\alpha){}^C_0D_t^{\alpha}u(x_i,t_{n-1})d\alpha\right)\\
%\nonumber
=&\frac{K}{2}\left(\frac{\partial^{\beta}u(x_i,t_n)}{\partial\mid
x\mid^{\beta}}+\frac{\partial^{\beta}u(x_i,t_{n-1})}{\partial\mid
x\mid^{\beta}}\right) +\frac{1}{2}\left(f(x_i,t_n)+f(x_i,t_{n-1})\right),
\end{split}
\label{2.2}
\end{equation}
where $1\leq i\leq M-1,~1\leq n\leq N$.

Define $U_i^n=u(x_i,t_n),~F_i^n=f(x_i,t_n)~(0\leq i\leq M,~0\leq n\leq N)$ on $\Omega_h\times\Omega_\tau$. Equation $\eqref{2.2}$ can be expressed as
\begin{align}
\int_1^2\omega(\alpha){}^C_0D_t^{\alpha}U_i^{n-\frac{1}{2}}d\alpha
=K\frac{\partial^{\beta}U_i^{n-\frac{1}{2}}}{\partial\mid x\mid^{\beta}}+F_i^{n-\frac{1}{2}},\quad 1\leq i\leq M-1,~1\leq n\leq N.\label{2.3}
\end{align}

Firstly, we consider the discretization of the integral term in $\eqref{2.3}$. Let $z(\alpha)=\omega(\alpha){}^C_0D_t^{\alpha}U_i^{n-\frac{1}{2}}$ and suppose that $\omega(\alpha)\in C^2([1,2])$, ${}^C_0D_t^{\alpha}u(x_i,t)\mid_{t=t_n}$ and ${}^C_0D_t^{\alpha}u(x_i,t)\mid_{t=t_{n-1}}\in C^2([1,2])$. According to Lemma $\ref{lemma2.1}$, we can obtain
\begin{equation*}
\int_1^2\omega(\alpha){}^C_0D_t^{\alpha}U_i^{n-\frac{1}{2}}d\alpha
=\Delta\alpha\sum\limits_{l=0}^{2J}c_l\omega(\alpha_l){}^C_0D_t^{\alpha_l}U_i^{n-\frac{1}{2}}
+R_1,
\end{equation*}
where $R_1=\mathcal{O}(\Delta\alpha^2)$. Thus the problem $\eqref{2.3}$ is now transformed into the following multi-term fractional diffusion-wave equation:
\begin{equation}
\Delta\alpha\sum\limits_{l=0}^{2J}c_l\omega(\alpha_l){}^C_0D_t^{\alpha_l}U_i^{n-\frac{1}{2}}
=K\frac{\partial^{\beta}U_i^{n-\frac{1}{2}}}{\partial\mid x\mid^{\beta}}+F_i^{n-\frac{1}{2}}+R_1,\quad 1\leq i\leq M-1,~1\leq n\leq N.\label{2.4}
\end{equation}

Next, we introduce a difference scheme to solve the multi-term system $\eqref{2.4}$ with the initial-boundary conditions $\eqref{1.2}$-$\eqref{1.3}$. Suppose $u(x,t)\in C^{(5,3)}([0,L]\times[0,T])$. For $1<\alpha_l<2$, using a fully discrete difference scheme $(2.7)$ in \cite{Gao2017Two} and noticing the zero initial condition $\eqref{1.2}$, we arrive at
\begin{align}
\nonumber
\Delta\alpha\sum\limits_{l=0}^{2J}c_l\omega(\alpha_l){}^C_0D_t^{\alpha_l}U_i^{n-\frac{1}{2}}
=&\Delta\alpha\sum\limits_{l=0}^{2J}c_l\omega(\alpha_l)_{-\infty}D_t^{\alpha_l}U_i^{n-\frac{1}{2}}\\
=&\Delta\alpha\sum\limits_{l=0}^{2J}c_l\omega(\alpha_l)\frac{1}{\tau^{\gamma_l}}\sum\limits_{k=0}^{n-1}
\lambda_k^{(\gamma_l)}\delta_tU_i^{n-k-\frac{1}{2}}+R_2, \label{2.5}
\end{align}
where $\gamma_l=\alpha_l-1~(0\leq l\leq2J)$, $R_2=\mathcal{O}(\tau^2)$ and
\begin{equation}\label{2.6}
\lambda_0^{(\gamma_l)}=(1+\frac{\gamma_l}{2})g_0^{(\gamma_l)};~~~~~
\lambda_k^{(\gamma_l)}=(1+\frac{\gamma_l}{2})g_k^{(\gamma_l)}-\frac{\gamma_l}{2}g_{k-1}^{(\gamma_l)},~k\geq1,
\end{equation}
with
\begin{equation*}
g_0^{(\gamma_l)}=1;~~~~~~~
g_k^{(\gamma_l)}=(1-\frac{\gamma_l+1}{k})g_{k-1}^{(\gamma_l)},~k\geq1.
\end{equation*}
In the meantime, using the Lemma $\ref{lemma2.2}$ for approximating the Riesz derivatives in $\eqref{2.4}$, we get
\begin{equation}
\frac{\partial^{\beta}U_i^{n-\frac{1}{2}}}{\partial| x|^{\beta}}=
-h^{-\beta}\sum\limits_{k=i-M}^i\hat{g}_k^{(\beta)}U_{i-k}^{n-\frac{1}{2}}+\mathcal{O}(h^2). \label{2.7}
\end{equation}
By substituting $\eqref{2.5}$ and $\eqref{2.7}$ into $\eqref{2.4}$, we obtain
\begin{equation}\label{2.8}
\Delta\alpha\sum\limits_{l=0}^{2J}c_l\omega(\alpha_l)\frac{1}{\tau^{\gamma_l}}\sum\limits_{k=0}^{n-1}
\lambda_k^{(\gamma_l)}\delta_tU_i^{n-k-\frac{1}{2}}
=-Kh^{-\beta}\sum\limits_{k=i-M}^i\hat{g}_k^{(\beta)}U_{i-k}^{n-\frac{1}{2}}+F_i^{n-\frac{1}{2}}+p_i^n,
\end{equation}
where $p_i^n=\mathcal{O}(\Delta\alpha^2+\tau^2+h^2)$, $1\leq i\leq M-1,~1\leq n\leq N$.

In addition, according to the initial and boundary value conditions, we obtain
\begin{align}
&U_i^0=0,\quad 1\leq i\leq M-1,\label{2.9}&\\
&U_0^n=0, \quad U_M^n=0,\quad 0\leq n\leq N.\label{2.10}&
\end{align}

Let $u_i^n$ be the numerical approximation to $u(x_i,t_n)$. Ignoring the error term $p_i^n$ in $\eqref{2.8}$, we can derive the following difference scheme for $\eqref{1.1}$-$\eqref{1.3}$
\begin{align}
\nonumber
&\Delta\alpha\sum\limits_{l=0}^{2J}c_l\omega(\alpha_l)\frac{1}{\tau^{\gamma_l}}\sum\limits_{k=0}^{n-1}
\lambda_k^{(\gamma_l)}\delta_tu_i^{n-k-\frac{1}{2}}
=-Kh^{-\beta}\sum\limits_{k=i-M}^i\hat{g}_k^{(\beta)}u_{i-k}^{n-\frac{1}{2}}+F_i^{n-\frac{1}{2}},\label{2.11}&\\
&\quad\qquad\quad\qquad\quad\qquad\quad\qquad\quad\qquad\quad\qquad 1\leq i\leq M-1,~1\leq n\leq N,&\\
&u_i^0=0,\quad 1\leq i\leq M-1,\label{2.12}&\\
&u_0^n=0, \quad u_M^n=0,\quad 0\leq n\leq N.\label{2.13}&
\end{align}

Let
$u^n=[u_1^n,v_2^n,\cdots,u_{M-1}^n]^T$ and $F^n=[F_1^n,F_2^n,\cdots,F_{M-1}^n]^T$.
Then the numerical scheme \eqref{2.11} can be rewritten into the following matrix form
\begin{equation}\label{2.14}
 Au^{n}=b^{n-1},\quad n=1,2,\ldots,N,%\tag{2.13}
\end{equation}
in which
 \begin{equation}\label{2.15}
 A=\mu_0I+K\nu_{\beta}G_\beta,
 \end{equation}
 and
 \begin{equation*}
 b^{n-1}=-K\nu_{\beta}G_\beta u^{n-1}
 +\sum\limits_{k=1}^{n-1}(\mu_{k-1}-\mu_k)u^{n-k}
 +\frac{\tau}{2}(F^n+F^{n-1}),
 \end{equation*}
 where $I$ is the identity matrix of order $M-1$ and $\nu_{\beta}=\frac{\tau}{2h^{\beta}}$,
 \begin{equation*}
 \mu_k=\Delta\alpha\sum\limits_{l=0}^{2J}c_l\omega(\alpha_l)\frac{1}{\tau^{\gamma_l}}
 \lambda_k^{(\gamma_l)},\quad k\geq0,
\end{equation*}
and
\begin{equation}\label{2.16}
G_\beta=
\begin{bmatrix}
{\hat{g}}_0^{(\beta)}&{\hat{g}}_{-1}^{(\beta)}&{\hat{g}}_{-2}^{(\beta)}&\cdots&{\hat{g}}_{3-M}^{(\beta)}&{\hat{g}}_{2-M}^{(\beta)}\\
{\hat{g}}_1^{(\beta)}&{\hat{g}}_{0}^{(\beta)}&{\hat{g}}_{-1}^{(\beta)}&\cdots&{\hat{g}}_{4-M}^{(\beta)}&{\hat{g}}_{3-M}^{(\beta)}\\
{\hat{g}}_2^{(\beta)}&{\hat{g}}_{1}^{(\beta)}&{\hat{g}}_{0}^{(\beta)}&\cdots&{\hat{g}}_{5-M}^{(\beta)}&{\hat{g}}_{4-M}^{(\beta)}\\
\vdots&\vdots&\vdots&\ddots&\vdots&\vdots\\
{\hat{g}}_{M-3}^{(\beta)}&{\hat{g}}_{M-4}^{(\beta)}&{\hat{g}}_{M-5}^{(\beta)}&\cdots&{\hat{g}}_{0}^{(\beta)}&{\hat{g}}_{-1}^{(\beta)}\\
{\hat{g}}_{M-2}^{(\beta)}&{\hat{g}}_{M-3}^{(\beta)}&{\hat{g}}_{M-4}^{(\beta)}&\cdots&{\hat{g}}_{1}^{(\beta)}&{\hat{g}}_{0}^{(\beta)}
\end{bmatrix}.
\end{equation}
Note that $G_\beta$ has a Toeplitz structure \cite{chan2007introduction}. According to Lemma \ref{lemma3.1} in Section \ref{section3}, the matrix $G_\beta$ is also symmetric. Therefore, it can be stored with only $M-1$ entries and the Toeplitz matrix-vector product can be performed within $\mathcal{O}(M\log M)$ operations using FFTs \cite{Jian2018A}.

\section{Solvability, stability and convergence analysis}\label{section3}
In this section, we discuss the solvability, stability and convergence of the difference scheme \eqref{2.11}-\eqref{2.13}. Before analyzing these properties, we first need to give some auxiliary definitions and useful lemmas as follows.

Denote the grid function space on $\Omega_h$ by $$V_h=\{v~|~v=(v_0,v_1,\cdots,v_{M-1},v_M)^T,~v_0=0,~v_M=0\}.$$

For any $v,~w\in V_h$, the discrete inner product and the corresponding discrete $L_2$-norm are defined as follows:
\begin{align*}
(v,w)=h\sum\limits_{i=1}^{M-1}v_iw_i,~~and~~ \|v\|=\sqrt{(v,v)}.
\end{align*}

%Lemma3.1 %\cite{ye2014numerical}
\begin{lemma} \cite{sun2015finite} Let $\hat{g}_k^{(\beta)}$ be defined as in Lemma \ref{lemma2.2}, for $1<\beta\leq2$, it satisfies
\begin{equation*}
\begin{cases}
\hat{g}_0^{(\beta)}=\frac{\Gamma(\beta+1)}{\Gamma^2(\beta/2+1)}\geq0,\quad \hat{g}_{-k}^{(\beta)}=\hat{g}_k^{(\beta)}\leq0,\quad k=1,2,\cdots,\\
\sum\limits_{k=-\infty}^{\infty}\hat{g}_k^{(\beta)}=0,\quad -\sum\limits_{k=-M+i\atop k\neq0}^{i}\hat{g}_k^{(\beta)}\leq \hat{g}_0^{(\beta)},\quad 1\leq i\leq M-1,\\
 \hat{g}_k^{(\beta)}=\left(1-\frac{\beta+1}{\beta/2+k}\right)\hat{g}_{k-1}^{(\beta)},\quad k\geq1,\\
 \sum\limits_{|k|=l}^{\infty}|\hat{g}_k^{(\beta)}|\geq\frac{c_*^{(\beta)}}{(l+1)^\beta},\quad l\geq1,
\end{cases}
\end{equation*}
where $c_*^{(\beta)}=\frac{2}{\beta}r_\beta$, with
$r_\beta=e^{-2}\frac{(4-\beta)(2-\beta)\beta}{(6+\beta)(4+\beta)(2+\beta)}\cdot
\frac{\Gamma(\beta+1)}{\Gamma^2(\frac{\beta}{2}+1)}\big(3+\frac{\beta}{2}\big)^{\beta+1}.$
\label{lemma3.1}
\end{lemma}

%Lemma3.2
\begin{lemma} \cite{Gao2017Two} Let $\{\lambda_k^{(\gamma)}\}_{k=0}^\infty~(0\leq\gamma\leq1)$ be defined in \eqref{2.6}, it holds that
\begin{equation*}
\begin{cases}
\lambda_0^{(\gamma)}=1+\frac{\gamma}{2}>0,\\
\lambda_1^{(\gamma)}=-\frac{1}{2}(\gamma+3)\gamma\leq0,\\
\lambda_2^{(\gamma)}=\frac{1}{4}(\gamma^2+3\gamma-2)\gamma=
\begin{cases}
\leq0,\quad \gamma\in[0,\frac{\sqrt{17}-3}{2}],\\
>0,\quad \gamma\in(\frac{\sqrt{17}-3}{2},1],
\end{cases}
\\
\lambda_k^{(\gamma)}=\left[(1+\frac{\gamma}{2})(1-\frac{1+\gamma}{k})
-\frac{\gamma}{2}\right]g_{k-1}^{(\gamma)}\leq0,\quad k=3,4,5,\cdot\cdot\cdot
\end{cases}
\end{equation*}
\label{lemma3.2}
\end{lemma}

%Lemma3.3
\begin{lemma} \cite{Gao2017Two} Let $\{\lambda_k^{(\gamma)}\}_{k=0}^\infty$ be defined in \eqref{2.6}. For any $\gamma\in[0,1]$, positive integer $m$ and real vector
$(v_1,v_2,\cdot\cdot\cdot,v_m)^T\in \mathbb{R}^m$, we have
\begin{equation*}
\sum\limits_{n=1}^{m}\left(\sum\limits_{k=0}^{n-1}\lambda_k^{(\gamma)}v_{n-k}\right)v_n\geq0.
\end{equation*}
\label{lemma3.3}
\end{lemma}

According to Lemma \ref{lemma3.1}, we get the following results.

%Lemma3.4%\cite{sun2015finite}
\begin{lemma}
Let $\delta_x^{\beta}=h^{-\beta}G_\beta~(1<\beta\leq2)$. For any grid function $v^n\in V_h~(0\leq n\leq N)$, we have
\begin{equation*}
\left(\delta_x^{\beta}v^{n-\frac{1}{2}},\delta_tv^{n-\frac{1}{2}}\right)
=\frac{1}{2\tau}\left(\parallel Bv^n\parallel^2-\parallel Bv^{n-1}\parallel^2\right),\quad 1\leq n\leq N,
\end{equation*}
where $B$ is a SPD matrix and satisfies $\delta_x^{\beta}=B^2$.
\label{lemma3.4}
\end{lemma}
\begin{proof}
According to Lemma \ref{lemma3.1}, it is easy to see that the $G_\beta$ is a SPD matrix, so $\delta_x^{\beta}$ is also SPD. Thus there is a SPD matrix $B$ such that $\delta_x^{\beta}=B^2$, and we have

\begin{align*}
\left(\delta_x^{\beta}v^{n-\frac{1}{2}},\delta_tv^{n-\frac{1}{2}}\right)
&=\left(B^2v^{n-\frac{1}{2}},\delta_tv^{n-\frac{1}{2}}\right)\\
%&=\left(B^2\frac{v^n+v^{n-1}}{2},\frac{v^n-v^{n-1}}{\tau}\right)\\
&=\frac{1}{2\tau}\left(Bv^n+Bv^{n-1},Bv^n-Bv^{n-1}\right)\\
&=\frac{1}{2\tau}\cdot h\sum\limits_{i=1}^{M-1}\left[(Bv_i^n)^2-(Bv_i^{n-1})^2\right]\\
&=\frac{1}{2\tau}\left(\parallel Bv^n\parallel^2-\parallel Bv^{n-1}\parallel^2\right),
\end{align*}
which completes the proof.
\end{proof}

%\cite{sun2015finite}
\begin{lemma}
 Let matrix $B$ as defined in Lemma \ref{lemma3.4}, For any grid function $v^n\in V_h~(0\leq n\leq N)$, we have
\begin{equation*}
\frac{c_*^{(\beta)}}{L^\beta}\parallel v^n\parallel^2\leq\parallel Bv^n\parallel^2
\leq \mid v^n\mid_{H^{\beta/2}}^2,
\end{equation*}
where $\mid v^n\mid_{H^{\beta/2}}^2$ is the fractional Sobolev norm, whose definition can be referred to \cite{Wang2015maximum}.
\label{lemma3.5}
\end{lemma}

\begin{proof}
According to \cite[Lemma 3.3]{Wang2015maximum}, we can easily get that $\parallel Bv^n\parallel^2
\leq \mid v^n\mid_{H^{\beta/2}}^2$.
%Since $\delta_x^{\beta}$ is the SPD Toeplitz matrix, its maximum eigenvalue $\lambda_1$ can be quickly calculated by the methods in \cite{Melman2001Extreme,Bad1998Parallel}.

Then we prove that $\parallel Bv^n\parallel^2>\frac{c_*^{(\beta)}}{L^\beta}\parallel v^n\parallel^2$. Consider the 2-norm of given vector,
\begin{align*}
\parallel Bv^n\parallel^2&=(Bv^n,Bv^n)\\
%&=h(v^n)^T(B^TB)v^n\\
&=h(v^n)^T(\delta_x^{\beta})v^n\\
&\stackrel{v^n=Py^n}{=\joinrel=\joinrel=\joinrel=}h(y^n)^T(P^T\delta_x^{\beta}P)y^n\\%\stackrel\overset
&=h(\lambda_1y_1^n+\lambda_2y_2^n+\cdot\cdot\cdot+\lambda_{M-1}y_{M-1}^n),
\end{align*}
where $v^n=Py^n$ is a orthogonal transformation, and $\lambda_1\geq\lambda_2\geq\cdot\cdot\cdot\geq\lambda_{M-1}$
are the eigenvalues of the matrix related to $\delta_x^{\beta}$.

Thus, we get
$$\parallel Bv^n\parallel^2\geq\lambda_{M-1}\parallel v^n\parallel^2.$$

So we just need to prove that $\lambda_{M-1}\geq\frac{c_*^{(\beta)}}{L^\beta}$.
Using Lemma \ref{lemma3.1}, we can easily get that the $\delta_x^{\beta}$ is
a strictly diagonally dominant $M$-matrix. According to \cite[Theorem 1.1]{Tian2013inequalities}, and noticing Lemma \ref{lemma3.1}, we have
\begin{align*}
\lambda_{M-1}&\geq h^{-\beta}\min\limits_{i\in\{1,2\cdot\cdot\cdot,M-1\}}\sum\limits_{k=i-M+1}^{i-1}\hat{g}_k^{(\beta)}\\
&\geq h^{-\beta}\left(\sum\limits_{|k|=0}^{M-2}\hat{g}_k^{(\beta)}\right)
%&= h^{-\beta}\left(-\sum\limits_{|k|=M-1}^{\infty}\hat{g}_k^{(\beta)}\right)\\
= h^{-\beta}\left(\sum\limits_{|k|=M-1}^{\infty}|\hat{g}_k^{(\beta)}|\right)\\
&\geq h^{-\beta}\left(\frac{c_*^{(\beta)}}{M^\beta}\right)
= \frac{c_*^{(\beta)}}{L^\beta}.
\end{align*}

The proof is completed.
\end{proof}

This lemma is the key to prove the stability and convergence of the difference scheme \eqref{2.11}-\eqref{2.13}.

\subsection{Solvability}\label{section3.1}
\begin{theorem}
The difference scheme \eqref{2.11}-\eqref{2.13} is uniquely solvable.
\label{th3.1}
\end{theorem}

\begin{proof}
By Lemma \ref{lemma3.2}, we have $\mu_0>0$. In addition, using Lemma \ref{lemma3.1}, noticing that $K>0$ and $\nu_\beta>0$, we can easily prove that the coefficient matrix $A$ defined in \eqref{2.15} is SPD, so it is nonsingular. Therefore the difference scheme \eqref{2.11}-\eqref{2.13} is uniquely solvable.
\end{proof}

\subsection{Stability}\label{section3.2}

In this subsection, we work towards proving the unconditional stability of the proposed difference scheme \eqref{2.11}-\eqref{2.13}.

\begin{theorem}
 Let $\{u_i^n~|~0\leq i\leq M,~0\leq n\leq N\}$ be the solution of the following difference system
\begin{align}
\nonumber
&\Delta\alpha\sum\limits_{l=0}^{2J}c_l\omega(\alpha_l)\frac{1}{\tau^{\gamma_l}}\sum\limits_{k=0}^{n-1}
\lambda_k^{(\gamma_l)}\delta_tu_i^{n-k-\frac{1}{2}}
=-Kh^{-\beta}\sum\limits_{k=i-M}^i\hat{g}_k^{(\beta)}u_{i-k}^{n-\frac{1}{2}}+G_i^{n},\label{3.1}&\\
&\quad\qquad\quad\qquad\quad\qquad\quad\qquad\quad\qquad\quad\qquad 1\leq i\leq M-1,~1\leq n\leq N,&\\
&u_i^0=\phi_i,\quad 1\leq i\leq M-1,\label{3.2}&\\
&u_0^n=0, \quad u_M^n=0,\quad 0\leq n\leq N.\label{3.3}&
\end{align}
Then it holds that
\begin{align*}
\parallel u^m\parallel^2\leq&
exp(T)\Bigg[\frac{3L^\beta}{c_{*}^{(\beta)}}\mid u^0\mid_{H^{\beta/2}}^2
+\left(\frac{2L^\beta}{Kc_{*}^{(\beta)}}\right)^2\left(\parallel G^1\parallel^2+\max\limits_{1\leq n\leq m}\parallel G^n\parallel^2\right) \\ &\quad\quad+\left(\frac{2L^\beta}{Kc_{*}^{(\beta)}}\right)^2\tau\sum\limits_{n=1}^{m-1}\parallel\delta_tG^{n+\frac{1}{2}}\parallel^2\Bigg],\quad 1\leq m\leq N,
\end{align*}
where
$\parallel G^n\parallel^2=h\sum\limits_{i=1}^{M-1}\left(G_i^{n}\right)^2$ and $
\parallel \delta_tG^{n+\frac{1}{2}}\parallel^2=h\sum\limits_{i=1}^{M-1}\left(\delta_tG_i^{n+\frac{1}{2}}\right)^2.$
\label{th3.2}
\end{theorem}

\begin{proof}
Taking the inner product of \eqref{3.1} with $\delta_tu^{n-\frac{1}{2}}$, we get
\begin{align*}
&\Delta\alpha\sum\limits_{l=0}^{2J}c_l\omega(\alpha_l)\frac{1}{\tau^{\gamma_l}}\sum\limits_{k=0}^{n-1}
\lambda_k^{(\gamma_l)}\left(\delta_tu^{n-k-\frac{1}{2}},\delta_tu^{n-\frac{1}{2}}\right)\\
=&-K\left(\delta_x^{\beta}u^{n-\frac{1}{2}},\delta_tu^{n-\frac{1}{2}}\right)
+\left(G^{n},\delta_tu^{n-\frac{1}{2}}\right),\quad 1\leq n\leq N.
\end{align*}
Using Lemma \ref{lemma3.4}, we get
\begin{align*}
&\Delta\alpha\sum\limits_{l=0}^{2J}c_l\omega(\alpha_l)\frac{1}{\tau^{\gamma_l}}\sum\limits_{k=0}^{n-1}
\lambda_k^{(\gamma_l)}\left(\delta_tu^{n-k-\frac{1}{2}},\delta_tu^{n-\frac{1}{2}}\right)\\
&+\frac{K}{2\tau}\left(\parallel Bu^n\parallel^2-\parallel Bu^{n-1}\parallel^2\right)
=\left(G^{n},\delta_tu^{n-\frac{1}{2}}\right),\quad 1\leq n\leq N.
\end{align*}
Summing up the above equality for $n$ from 1 to $m$, we have
\begin{align*}
&\Delta\alpha\sum\limits_{l=0}^{2J}c_l\omega(\alpha_l)\frac{1}{\tau^{\gamma_l}}
\sum\limits_{n=1}^{m}\sum\limits_{k=0}^{n-1}
\lambda_k^{(\gamma_l)}\left(\delta_tu^{n-k-\frac{1}{2}},\delta_tu^{n-\frac{1}{2}}\right)\\
&+\frac{K}{2\tau}\left(\parallel Bu^m\parallel^2-\parallel Bu^{0}\parallel^2\right)
=\sum\limits_{n=1}^{m}\left(G^{n},\delta_tu^{n-\frac{1}{2}}\right),\quad 1\leq m\leq N.
\end{align*}
According to Lemma \ref{lemma3.3}, we obtain
\begin{align*}
\Delta\alpha\sum\limits_{l=0}^{2J}c_l\omega(\alpha_l)\frac{1}{\tau^{\gamma_l}}
\sum\limits_{n=1}^{m}\sum\limits_{k=0}^{n-1}
\lambda_k^{(\gamma_l)}\left(\delta_tu^{n-k-\frac{1}{2}},\delta_tu^{n-\frac{1}{2}}\right)\geq0.
\end{align*}
Consequently, we can get
\begin{align*}
%\frac{1}{2\tau}\left(\frac{Kc_*^{(\beta)}}{L^{\beta}}\parallel u^m\parallel^2-2K\hat{g}_0^{(\beta)}\parallel u^{0}\parallel^2\right)
\frac{K}{2\tau}\left(\| Bu^m\|^2-\| Bu^{0}\|^2\right)
\leq&\sum\limits_{n=1}^{m}\left(G^{n},\delta_tu^{n-\frac{1}{2}}\right)\\
=&\frac{1}{\tau}\left[\sum\limits_{n=1}^{m}(G^n,u^n)-\sum\limits_{n=1}^{m}(G^n,u^{n-1})\right]\\
=&-\sum\limits_{n=1}^{m-1}(\frac{G^{n+1}-G^n}{\tau},u^n)-\frac{1}{\tau}(G^1,u^0)+
\frac{1}{\tau}(G^m,u^m)\\
\leq&\sum\limits_{n=1}^{m-1}\left[\frac{L^{\beta}}{Kc_*^{(\beta)}}\parallel \delta_tG^{n+\frac{1}{2}}\parallel^2+\frac{Kc_*^{(\beta)}}{4L^{\beta}}\parallel u^n\parallel^2\right]\\
&+\frac{1}{\tau}\left[\frac{L^{\beta}}{Kc_*^{(\beta)}}\parallel G^1\parallel^2+\frac{Kc_*^{(\beta)}}{4L^{\beta}}\parallel u^0\parallel^2\right]\\
&+\frac{1}{\tau}\left[\frac{L^{\beta}}{Kc_*^{(\beta)}}\parallel G^m\parallel^2+\frac{Kc_*^{(\beta)}}{4L^{\beta}}\parallel u^m\parallel^2\right]\\
\leq&\sum\limits_{n=1}^{m-1}\left[\frac{L^{\beta}}{Kc_*^{(\beta)}}\parallel \delta_tG^{n+\frac{1}{2}}\parallel^2+\frac{K}{4}\parallel Bu^n\parallel^2\right]\\
&+\frac{1}{\tau}\left[\frac{L^{\beta}}{Kc_*^{(\beta)}}\parallel G^1\parallel^2+\frac{K}{4}\parallel Bu^0\parallel^2\right]\\
&+\frac{1}{\tau}\left[\frac{L^{\beta}}{Kc_*^{(\beta)}}\parallel G^m\parallel^2+\frac{K}{4}\parallel Bu^m\parallel^2\right],\quad 1\leq m\leq N,
\end{align*}
where Lemma \ref{lemma3.5} applies to the last step.
Then
\begin{align*}
\parallel Bu^m\parallel^2\leq&
3\parallel Bu^0\parallel^2
+\frac{4L^\beta}{K^2c_{*}^{(\beta)}}\left(\parallel G^1\parallel^2+\max\limits_{1\leq n\leq m}\parallel G^n\parallel^2\right)\\ &+\frac{4L^\beta}{K^2c_{*}^{(\beta)}}\tau\sum\limits_{n=1}^{m-1}\parallel\delta_tG^{n+\frac{1}{2}}\parallel^2
+\tau\sum\limits_{n=1}^{m-1}\parallel Bu^n\parallel^2,\quad 1\leq m\leq N.
\end{align*}
Applying the Gronwall inequality \cite[Lemma 2]{Sun2013appliedmm} for above equation, we obtain
\begin{align*}
\parallel Bu^m\parallel^2\leq&
exp(T)\Bigg[3\parallel Bu^0\parallel^2
+\frac{4L^\beta}{K^2c_{*}^{(\beta)}}\left(\parallel G^1\parallel^2+\max\limits_{1\leq n\leq m}\parallel G^n\parallel^2\right) \\ &\quad\quad+\frac{4L^\beta}{K^2c_{*}^{(\beta)}}\tau\sum\limits_{n=1}^{m-1}\parallel\delta_tG^{n+\frac{1}{2}}\parallel^2\Bigg],\quad 1\leq m\leq N.
\end{align*}
Then by Lemma \ref{lemma3.5}, it arrives at
\begin{align*}
\parallel u^m\parallel^2\leq&
exp(T)\Bigg[\frac{3L^\beta}{c_{*}^{(\beta)}}\mid u^0\mid_{H^{\beta/2}}^2
+\left(\frac{2L^\beta}{Kc_{*}^{(\beta)}}\right)^2\left(\parallel G^1\parallel^2+\max\limits_{1\leq n\leq m}\parallel G^n\parallel^2\right) \\ &\quad\quad+\left(\frac{2L^\beta}{Kc_{*}^{(\beta)}}\right)^2\tau\sum\limits_{n=1}^{m-1}\parallel\delta_tG^{n+\frac{1}{2}}\parallel^2\Bigg],\quad 1\leq m\leq N.
\end{align*}
The proof completed.
\end{proof}

\subsection{Convergence}\label{section3.3}
Now, we consider the convergence of the proposed difference scheme \eqref{2.11}-\eqref{2.13}. The following theorem verifies the second-order accuracy of our proposed difference scheme in time, space and distributed order.
\begin{theorem}
 Suppose that Eqs. \eqref{1.1}-\eqref{1.3} has a sufficiently smooth solution $u(x,t)\in C_{x,t}^{5,3}([0,L]\times[0,T])$. Let $U_i^n=u(x_i,t_n)$ be the exact solution of Eqs. \eqref{1.1}-\eqref{1.3} and $u_i^n$ be the numerical approximation  of the difference scheme \eqref{2.11}-\eqref{2.13}. The error $e_i^n=U_i^n-u_i^n~(0\leq i\leq M,~0\leq n\leq N$).
 Then there is a positive constant $C$ such that the error satisfies
\begin{align*}
\parallel e^n\parallel
\leq
\frac{2L^\beta}{Kc_{*}^{(\beta)}}\sqrt{exp(T)L(2+T)}C(\Delta\alpha^2+\tau^2+h^2),
\quad 1\leq n\leq N.
\end{align*}
\label{th3.3}
\end{theorem}

\begin{proof}
Subtracting \eqref{2.11}-\eqref{2.13} from \eqref{2.8}-\eqref{2.10}, we can get
\begin{align}
\nonumber
&\Delta\alpha\sum\limits_{l=0}^{2J}c_l\omega(\alpha_l)\frac{1}{\tau^{\gamma_l}}\sum\limits_{k=0}^{n-1}
\lambda_k^{(\gamma_l)}\delta_te_i^{n-k-\frac{1}{2}}
=-Kh^{-\beta}\sum\limits_{k=i-M}^i\hat{g}_k^{(\beta)}e_{i-k}^{n-\frac{1}{2}}+p_i^n,&\\
&\quad\qquad\quad\qquad\quad\qquad\quad\qquad\quad\qquad\quad\qquad 1\leq i\leq M-1,~1\leq n\leq N,\label{3.4}&\\
&e_i^0=0,\quad 1\leq i\leq M-1,\label{3.5}&\\
&e_0^n=0, \quad e_M^n=0,\quad 0\leq n\leq N.\label{3.6}&
\end{align}
Since $p_i^n=\mathcal{O}(\Delta\alpha^2+\tau^2+h^2)$, there exists a positive constant $C$, such that
\begin{align}
&\mid p_i^n\mid\leq C(\Delta\alpha^2+\tau^2+h^2),\label{3.7}\\
&\mid \delta_tp_i^{n-\frac{1}{2}}\mid\leq C(\Delta\alpha^2+\tau^2+h^2).\label{3.8}
\end{align}
%where $1\leq i\leq M-1,~1\leq n\leq N.$

Using Theorem \ref{th3.2} and \eqref{3.7}-\eqref{3.8}, and noticing the \eqref{3.5}-\eqref{3.6}, we obtain
\begin{equation*}
\begin{split}
\|e^m\|^2 \leq&
exp(T)\Bigg[\frac{3L^\beta}{c_{*}^{(\beta)}}\mid u^0\mid_{H^{\beta/2}}^2+
\left(\frac{2L^\beta}{Kc_{*}^{(\beta)}}\right)^2\left(\parallel p^1\parallel^2+
\max\limits_{1\leq n\leq m}\parallel p^n\parallel^2+
\tau\sum\limits_{n=1}^{m-1}\parallel\delta_tp^{n+\frac{1}{2}}\parallel^2\right)\Bigg]\\
\leq&
exp(T)\left(\frac{2L^\beta}{Kc_{*}^{(\beta)}}\right)^2
\left[2LC^2(\Delta\alpha^2+\tau^2+h^2)^2+\tau (m-1)LC^2(\Delta\alpha^2+\tau^2+h^2)^2\right]\\
\leq&
exp(T)\left(\frac{2L^\beta}{Kc_{*}^{(\beta)}}\right)^2(2+T)LC^2(\Delta\alpha^2+\tau^2+h^2)^2,
\quad 1\leq m\leq N.
\end{split}
\end{equation*}
Further, if follows that
\begin{align*}
\parallel e^m\parallel
\leq
\frac{2L^\beta}{Kc_{*}^{(\beta)}}\sqrt{exp(T)L(2+T)}C(\Delta\alpha^2+\tau^2+h^2),
\quad 1\leq m\leq N.
\end{align*}
This completes the proof.
\end{proof}

\section{PKSM-based GSF}\label{section4}
%GSF based on PCG method with T. Chan's circulant preconditioner
In this section, the implementation of the numerical method \eqref{2.14} is analyzed. Based on the SPD Toeplitz structure of the coefficient matrix, we will design the GSF utilizing the preconditioned CG (PCG-based GSF) method with a circulant preconditioner to solve the linear system \eqref{2.14}.

In general, the implementation of the linear system \eqref{2.14} can be represented by the following algorithm.
\begin{algorithm}[H]
\caption{Practical implementation of \eqref{2.14}}
\label{alg1}
\begin{algorithmic}[1]
\small
% 1
\STATE {\textbf{for} $n=1,2,\cdot\cdot\cdot,N$, \textbf{do}}
% 2
\STATE\quad {Compute $b^{n-1}=-K\nu_{\beta}G_\beta u^{n-1}
 +\sum\limits_{k=1}^{n-1}(\mu_{k-1}-\mu_k)u^{n-k}
 +\frac{\tau}{2}(F^n+F^{n-1})$}
% 3
\STATE\quad {Solve $Au^{n}=b^{n-1}$}
% 4
\STATE {\textbf{end for}}
\end{algorithmic}
\end{algorithm}
In Algorithm \ref{alg1}, by exploiting the Toeplitz structure of the matrix $G_\beta$, the matrix-vector product $G_\beta u^{n-1}$ in Step 2 can be done in $\mathcal{O}(M\log M)$ operations by FFTs.
Since the coefficient matrix $A$ is time-independent (it does not depend on $n$), the solutions of \eqref{2.14} in Step 3 can be written as $u^n=A^{-1}b^{n-1}~(n=1,2,\cdot\cdot\cdot,N$) and can be calculated at the cost of one Cholesky factorization. However, the method requires considerable cost when $A$ is large and dense. Fortunately, $A$ has a SPD Toeplitz structure, so $A^{-1}b$ can be directly calculated by GSF \cite{Pang2011Shift,Li2018A,gu2016fast} with limited memory and computational cost. More precisely, let $l$=$(l_1,l_2,\cdot\cdot\cdot,l_{M-1})^T$ be the solution of the linear system
\begin{equation}\label{4.1}
Al=e_1=(1,0,\cdot\cdot\cdot,0)^T.
\end{equation}
Then the matrix-vector multiplication $A^{-1}b$ by the GSF is formulated as
$$A^{-1}b=Re(z)+JIm(z),$$
where $J$ is a $M-1$ dimensional anti-identity matrix, $Re(z)$ and $Im(z)$ represent the real and imaginary parts of $z$, respectively, and
\begin{equation}\label{4.2}
z=\frac{1}{2l_1}\left[(L+\hat{L}^T)(L^T-\hat{L})\right](b+\textbf{i}Jb)
\end{equation}
%$$A^{-1}=\frac{1}{l_1}(LL^T-\hat{L}\hat{L}^T),$$
with $L$ and $\hat{L}$ are lower Toeplitz matrices which are defined as

\begin{equation*}
L=
\begin{bmatrix}
{l_1}&0&\cdots&0\\
l_2&l_1&\ddots&\vdots\\
\vdots&\ddots&\ddots&0\\
l_{M-1}&\cdots&l_2&l_1
\end{bmatrix}
\quad
\rm{and}
\quad
\hat{L}=
\begin{bmatrix}
0&0&\cdots&0\\
l_{M-1}&0&\ddots&\vdots\\
\vdots&\ddots&\ddots&0\\
l_2&\cdots&l_{M-1}&0
\end{bmatrix},
\end{equation*}
respectively.

We note that in \eqref{4.2}, $L+\hat{L}^T$ is a circulant matrix \cite{lei2013circulant} and $L^T-\hat{L}$ is a skew-circulant \cite{Ng2003Circulant} matrix. Therefore, the cost of computing the $A^{-1}b$ is almost the same as computing one circulant and one skew-circulant matrix-vector product, or roughly only four $(M-1)$-length FFTs with $\mathcal{O}(M\log M)$ complexity. The algorithm to compute the product $A^{-1}b$ via GSF is given in Algorithm \ref{alg2} as follows:

\begin{algorithm}[H]
\caption{GSF for solving $z=A^{-1}b$}
\label{alg2}
\begin{algorithmic}[1]
\small
% 1
\STATE {Solve the linear system $Al=e_1$ in \eqref{4.1}}
% 2
\STATE {Compute $z=(L^T-\hat{L})(b+\textbf{i}Jb)$ via FFTs ($\mathbf{i} = \sqrt{-1}$)}
% 3
\STATE {Compute $z=\frac{1}{2l_1}(L+\hat{L}^T)z$ via FFTs}
% 4
\STATE {Compute $z=Re(z)+JIm(z)$}
\end{algorithmic}
\end{algorithm}

In summary, once the $l$ in \eqref{4.1} is obtained, the matrix-vector product $A^{-1}b$ can be done in $\mathcal{O}(M\log M)$ operations with Algorithm \ref{alg2}. Next, an efficient iterative algorithm will be developed to solve the SPD Toeplitz linear system \eqref{4.1}.

 %\cite{lei2013circulant}.
It is well-known that the CG method \cite{chan1989toeplitz} is a popular and effective KSM for solving the SPD linear systems. However, the CG method usually converges slowly when the coefficient matrix $A$ is ill-conditioned or large. Therefore, the preconditioned CG (PCG) method is proposed to solve the linear system \eqref{4.1}, whose computational complexity at each time step is only $\mathcal{O}(M\log M)$.

Now, a circulant preconditioner, which is generated from the R. Chan's
\cite{Ng2004Iterative} preconditioner, is proposed to solve the linear system
\eqref{4.1}. Note that other circulant approximations, such as the Strang's preconditioner \cite{chan2007introduction}, are also available but will not be discussed here.
For a Toeplitz matrix $G\in\mathbb{C}^{n\times n}$ defined as in \eqref{2.16}, the R. chan's preconditioner $r(G)$ is a circulant matrix obtained by using all entries of $G$ \cite{chan2007introduction}. More precisely, its entries $r_{ij}=r_{i-j}$ are given by
\begin{eqnarray*}
r_k=
\begin{cases}
\hat{g}_0,    &k=0,\\
\hat{g}_k+\hat{g}_{k-n},      &0<k<n\\
r_{k+n},&0<-k<n.
\end{cases}
\end{eqnarray*}
Then, solving \eqref{4.1} is equivalent to solving the following preconditioned system
\begin{equation*}
 R^{-1}Al=R^{-1}e_1,%\tag{2.13}
\end{equation*}
where the R. Chan's-based circulant preconditioner $R$ is defined as
\begin{equation}\label{4.3}
 R=\mu_0I+K\nu_{\beta}r(G_\beta).
 %R=\frac{\mu_0}{\tau}I+\frac{Kh^{-\beta}}{2}r(G_\beta).
 \end{equation}
More precisely, the first column of $r(G_\beta)$ is given by
 \begin{center}
$\left(
  \begin{array}{c}
    \hat{g}_0^{(\beta)} \\
   \hat{g}_1^{(\beta)}+\hat{g}_{2-M}^{(\beta)} \\
   \hat{g}_2^{(\beta)}+\hat{g}_{3-M}^{(\beta)} \\
    \vdots \\
    \vdots \\
    \hat{g}_{M-3}^{(\beta)}+\hat{g}_{-2}^{(\beta)} \\ \\
    \hat{g}_{M-2}^{(\beta)}+\hat{g}_{-1}^{(\beta)} \\ \\
  \end{array}
\right)$.
 \end{center}

%The following lemma guarantees that the circulant preconditioner $R$ is nonsingular.
%
% %\begin{lemma}
% The circulant preconditioner
%\begin{equation*}
%  R=\mu_0I+K\nu_{\beta}r(G_\beta),
%  %R=\frac{\mu_0}{\tau}I+\frac{Kh^{-\beta}}{2}r(G_\beta)
% \end{equation*}
%is symmetric positive definite.
%\label{lemma4.1}
%\end{lemma}
%
%\begin{proof}
%Suppose that $r_{ij}^n$ is the $(i,j)$ entry of $R$. Since $\mu_0>0,~K>0,~\nu_{\beta}>0$, and thanks to Lemma \ref{lemma3.1}, we have
%\begin{equation*}
%\begin{split}
%& |r_{ii}^n|-\sum\limits_{j=1,j\neq i}^{M-1}|r_{ij}^n| \\
% =& \left(\mu_0+K\nu_{\beta}\hat{g}_0^{(\beta)}\right)-
% K\nu_{\beta}\left(\sum\limits_{j=1}^{M-2}|\hat{g}_j^{(\beta)}+\hat{g}_{j-M+1}^{(\beta)}|\right) \\
%=&\mu_0+K\nu_{\beta}\sum\limits_{j=2-M}^{M-2}\hat{g}_j^{(\beta)}\\
%>&\mu_0>0,
%\end{split}
%\end{equation*}
%which implies that $R$ is strictly diagonally dominant. Noting that $R$ is symmetric and its main diagonal elements are positive, thus we get that the matrix $R$ is symmetric positive definite.
%\end{proof}The following lemma is essential to prove that the preconditioner $R$ is nonsingular.

To analyze the properties of the preconditioner $R$, we first prove the following lemma.
\begin{lemma}
For $1<\beta\leq2$, the circulant matrix $r(G_\beta)$ is real SPD, and all its eigenvalues fall inside $(0, 2\hat{g}_0^{(\beta)})$.
\label{lemma4.1}
\end{lemma}

\begin{proof}
According to Lemma \ref{lemma3.1}, it is easy to obtain that $r(G_\beta)$ is a real symmetric matrix, and all the Gershgorin disc \cite{lei2013circulant} of the circulant matrix $r(G_\beta)$ are centered at $\hat{g}_0^{(\beta)}$ with radius
$$r=\sum\limits_{j=1}^{M-2}|\hat{g}_j^{(\beta)}+\hat{g}_{j-M+1}^{(\beta)}|
   =-\sum\limits_{|j|=1}^{M-2}\hat{g}_j^{(\beta)}
   <-\sum\limits_{|j|=1}^{\infty}\hat{g}_j^{(\beta)}
   =\hat{g}_0^{(\beta)}.$$
Based on the Gershgorin theorem, we get that all the eigenvalues of $r(G_\beta)$ fall inside $(0, 2\hat{g}_0^{(\beta)})$. Therefore, the matrix $r(G_\beta)$ is real SPD.
\end{proof}

Then we can conclude that $R$ has the following properties.

\begin{theorem}
For $1<\beta\leq2$, the circulant preconditioner $R$ in \eqref{4.3} is SPD and
\begin{equation*}
\parallel R^{-1}\parallel_2<\frac{1}{\mu_0},\quad
\parallel R\parallel_2<\mu_0+2K\nu_{\beta}\hat{g}_0^{(\beta)}.
\end{equation*}
\label{th4.1}
\end{theorem}

\begin{proof}
By Lemma \ref{lemma4.1} and noting that $\mu_0>0$, $K>0$ and $\nu_\beta>0$, the circulant preconditioner $R$ is also a SPD matrix.

Suppose that $\{\lambda_k(r(G_\beta))|k=1,2,\cdot\cdot\cdot,M-1\}$ are the eigenvalues of $r(G_\beta)$. According  to Lemma \ref{lemma4.1}, we have $0<\lambda_k(r(G_\beta))<2\hat{g}_0^{(\beta)}$ for each $k=1,2,\cdot\cdot\cdot,M-1$. Since the $k$-th eigenvalue of $R$ is
\begin{equation*}
\lambda_k(R)=\mu_0+K\nu_{\beta}\lambda_k(r(G_\beta)),
\end{equation*}
we get $\mu_0<\lambda_k(R)<\mu_0+2K\nu_{\beta}\hat{g}_0^{(\beta)}$ for all $k=1,2,\cdot\cdot\cdot,M-1$.
Furthermore, we have
\begin{equation*}
\parallel R^{-1}\parallel_2=\frac{1}{\min\limits_{1\leq k\leq M-1}|\lambda_k(R)|}
<\frac{1}{\mu_0},
\end{equation*}
and
\begin{equation*}
\parallel R\parallel_2=\max\limits_{1\leq k\leq M-1}|\lambda_k(R)|
<\mu_0+2K\nu_{\beta}\hat{g}_0^{(\beta)},
\end{equation*}
which completes the proof.
\end{proof}

This theorem will be applied to prove the superlinear convergence rate of PCG method with R. Chan's circulant preconditioner in the next section.

\section{Spectrum of the preconditioned matrix}\label{section5}
In this section, the spectral properties of the preconditioned matrix $R^{-1}A$ are analyzed.
The eigenvalue distribution of the preconditioned matrix is one of the key factors that affect the convergence rate of the KSMs \cite{Ng2004Iterative}. Usually, the convergence speed is fast if the spectrum of the preconditioned matrix is away from zero, and the preconditioned matrix can be expressed as the sum of an identity matrix, a matrix with small norm, and a matrix with low rank, especially for those matrices close to normal \cite{Pan2014Preconditioning,Benzi2002Preconditioning}.
For convenience, we choose proper $N$, depending on $M$, and find a $\nu\in\mathbb{R}$ such that
\begin{equation}\label{5.1}
\nu_{\beta}\leq\nu
\end{equation}
for all $M$.

%In order to analyze the spectrum of the preconditioned matrix $R^{-1}A$, we first need to introduce the definition of the generating function.
%\begin{definition} \cite{Chan1996Conjugate,lei2013circulant}
%The generating function of the sequence of Toeplitz matrix $\{A_M\}_{M=1}^\infty$ is
%\begin{equation*}
%f(\theta)=\sum\limits_{k=-\infty}^{\infty}a_ke^{\textbf{i}k\theta},
%\end{equation*}
%where $a_k$ is the $k$-th diagonal element of $A_M$. The generating function $f(\theta)$ is in the Wiener class if and only if
%$
%\sum\limits_{k=-\infty}^{\infty}|a_k|<\infty.
%$
%\label{def5.1}
%\end{definition}

By the definition of generating function \cite[pp.12-20]{chan2007introduction}, we can obtain that the generating function of $G_\beta$ is
\begin{equation*}
p(\theta)=\sum\limits_{k=-\infty}^{\infty}\hat{g}_k^{(\beta)}e^{\textbf{i}k\theta}.
%=\sum\limits_{k=0}^{\infty}\hat{g}_0^{(\beta)}e^{\textbf{i}k\theta}.
\end{equation*}
We remark that the $\mu_0+K\nu_{\beta}p(\theta)$ cannot be a generating function of $A$ in \eqref{2.15} since $\nu_{\beta}$ is dependent on $M$. In order to prove the superlinear convergence of the preconditioned matrix $R^{-1}A$, the following two lemmas are needed.

\begin{lemma}
Let $p(\theta)$ be the generating function of $\{G_{\beta}\}_{M=1}^\infty$, then we get that the $p(\theta)$ is in the Wiener class.
\label{lemma5.1}
\end{lemma}
\begin{proof}
By Lemma \ref{lemma3.1}, we get
%\begin{equation*}
$\sum\limits_{k=-\infty}^{\infty}|\hat{g}_k^{(\beta)}|
=\hat{g}_0^{(\beta)}-\sum\limits_{|k|=1}^{\infty}\hat{g}_k^{(\beta)}
=2\hat{g}_0^{(\beta)}
<\infty,$
%\end{equation*}
which proves that $p(\theta)$ is in the Wiener class \cite[pp.17-20]{chan2007introduction}.
\end{proof}

\begin{lemma} \cite{Chan1996Conjugate}
If $p(\theta)$, the generating function of $\{G_{\beta}\}_{M=1}^\infty$, is in the Wiener class, then for any
$\varepsilon>0$, there exist $M'$ and $N'>0$, such that for all $M>M'$, we have
\begin{equation*}
G_{\beta}-r(G_{\beta})=U_M+V_M
\end{equation*}
where $rank(U_M)\leq N'$ and $\parallel V_M\parallel_2<\frac{\mu_0}{K\nu}\varepsilon$.
\label{lemma5.2}
\end{lemma}

\begin{theorem}
If $1<\beta\leq2$ and $\nu_{\beta}\leq\nu$, for any $\varepsilon>0$, there exist $M^*$ and $N^*>0$ such that, for any $M>M^*$,we have
\begin{equation*}
R^{-1}A-I=\hat{U}_M+\hat{V}_M,
\end{equation*}
where $rank(\hat{U}_M)\leq N^*$ and $\|\hat{V}_M\|_2 < \varepsilon$.
\label{th5.1}
\end{theorem}
\begin{proof}
According to Lemma \ref{lemma5.2}, we have
\begin{align*}
R^{-1}A-I&=R^{-1}(A-R)\\
%&=R^{-1} [(\mu_0I+K\nu_{\beta}G_\beta)-(\mu_0I+K\nu_{\beta}r(G_\beta))]\\
&=K\nu_{\beta}R^{-1}[G_\beta-r(G_\beta)]\\
&=K\nu_{\beta}R^{-1}(U_M+V_M)\\
&=K\nu_{\beta}R^{-1}U_M+K\nu_{\beta}R^{-1}V_M,
\end{align*}
where
\begin{align*}
rank\left(K\nu_{\beta}R^{-1}U_M\right)&=rank\left(R^{-1}U_M\right)\\
&\leq rank(U_M)\\
&\leq N'.
\end{align*}
Therefore, let $\hat{U}_M=K\nu_{\beta}R^{-1}U_M$ and $N^*=N'$, we have
$rank(\hat{U}_M)\leq N^*.$

Let $\hat{V}_M=K\nu_{\beta}R^{-1}V_M$. By Theorem \ref{th4.1} and Lemma \ref{lemma5.2}, we get
\begin{align*}
\parallel \hat{V}_M\parallel_2
&\leq K\nu_{\beta}\parallel R^{-1}\parallel_2\cdot\parallel V_M\parallel_2\\
&< K\nu\frac{1}{\mu_0}\cdot\frac{\mu_0}{K\nu}\varepsilon\\
&=\varepsilon.
\end{align*}
This completes the proof.
\end{proof}

From Theorem \ref{th5.1}, it holds that the preconditioned matrix is the sum of an identity matrix, a matrix with small norm, and a matrix with low rank. Thus the spectrum of $R^{-1}A$ is clustered around 1.
Furthermore, since $\sigma_{\rm{min}}(A)>\mu_0$ is clearly satisfied, and according to the Theorem \ref{th4.1} and Eq. \eqref{5.1}, the smallest singular value of the matrix $R^{-1}A$ is
\begin{equation*}
\sigma_{\rm{min}}(R^{-1}A)\geq\frac{\sigma_{\rm{min}}(A)}{\parallel R\parallel_2}
>\frac{\mu_0}{\mu_0+2K\nu_{\beta}\hat{g}_0^{(\beta)}}
\geq\frac{\mu_0}{\mu_0+2K\nu\hat{g}_0^{(\beta)}}
>0,
\end{equation*}
which implies $\sigma_{\rm{min}}(R^{-1}A)$ is uniformly bounded away from zero. According to the Corollary 1.11 in \cite{chan2007introduction}, the superlinear convergence of the PCG method with R. Chan's circulant preconditioner is obtained.

\section{2D problem}\label{section6}
In this section, we consider the following 2D problem:
\begin{align}
\nonumber
&\int_1^2\omega(\alpha){}^C_0D_t^{\alpha}u(x,y,t)d\alpha=K_1\frac{\partial^{\beta}u(x,y,t)}{\partial\mid
x\mid^{\beta}}+K_2\frac{\partial^{\gamma}u(x,y,t)}{\partial\mid
y\mid^{\gamma}}+f(x,y,t),\\
&\qquad\qquad\qquad\qquad\qquad\qquad\qquad\qquad\quad (x,y)\in\Omega,~0<t\leq T,&\label{6.1}\\
&u(x,y,0)=0,\quad u_t(x,y,0)=0,\quad (x,y)\in\Omega,\label{6.2}&\\
&u(x,y,t)=0,\quad (x,y)\in\partial\Omega,\quad 0\leq t\leq T,&\label{6.3}
\end{align}
where $K_1,K_2>0$, $\gamma\in(1,2]$, $\Omega=[0,L_1]\times[0,L_2]$, $\partial\Omega$ is the boundary of $\Omega$, and $f(x,y,t)$ is a given function.

\subsection{The derivation of the difference scheme}\label{section6.1}
The numerical discrete methods applied to the 1D problem \eqref{1.1}-\eqref{1.3} can be directly extended to deal with the 2D problem \eqref{6.1}-\eqref{6.3}.

To derive the difference scheme of \eqref{6.1}-\eqref{6.3}, we first divide the interval $[0,L_1]$ into $M_1$-subintervals with $h_1=\frac{L_1}{M_1}$ and $x_i=ih_1~(0\leq i\leq M_1)$, and divide the interval $[0,L_2]$ into $M_2$-subintervals with $h_2=\frac{L_2}{M_2}$ and $y_j=jh_2~(0\leq j\leq M_2)$.

Let $\omega=\{(i,j)~|1\leq i\leq M_1-1,~1\leq j\leq M_2-1\}$, $\partial\omega=\{(i,j)~|~(x_i,y_j)\in\partial\Omega\}$, $\bar{\omega}=\omega\bigcup\partial\omega$. Suppose $u(x,y,t)\in C^{(5,5,3)}(\bar{\Omega}\times[0,T])$ and consider Eq. \eqref{6.1} at the point $(x_i,y_j,t_n)$, we have
\begin{equation*}
\int_1^2\omega(\alpha){}^C_0D_t^{\alpha}u(x_i,y_j,t_n)d\alpha=K_1\frac{\partial^{\beta}u(x_i,y_j,t_n)}{\partial\mid
x\mid^{\beta}}+K_2\frac{\partial^{\gamma}u(x_i,y_j,t_n)}{\partial\mid
y\mid^{\gamma}}+f(x_i,y_j,t_n),
\end{equation*}
where $(i,j)\in\omega,~0\leq n\leq N$. Taking an average of the above equality on time levels $t=t_n$ and $t=t_{n-1}$, it follows
\begin{align}
\nonumber
&\frac{1}{2}\left(\int_1^2\omega(\alpha){}^C_0D_t^{\alpha}u(x_i,y_j,t_n)d\alpha+
\int_1^2\omega(\alpha){}^C_0D_t^{\alpha}u(x_i,y_j,t_{n-1})d\alpha\right)\\
\nonumber
=&\frac{K_1}{2}\left(\frac{\partial^{\beta}u(x_i,y_j,t_n)}{\partial\mid
x\mid^{\beta}}+\frac{\partial^{\beta}u(x_i,y_j,t_{n-1})}{\partial\mid
x\mid^{\beta}}\right)+
\frac{K_2}{2}\left(\frac{\partial^{\gamma}u(x_i,y_j,t_n)}{\partial\mid
y\mid^{\gamma}}+\frac{\partial^{\gamma}u(x_i,y_j,t_{n-1})}{\partial\mid
y\mid^{\gamma}}\right)\\
&+\frac{1}{2}\left(f(x_i,y_j,t_n)+f(x_i,y_j,t_{n-1})\right),\quad (i,j)\in\omega,~1\leq n\leq N.\label{6.4}
\end{align}

Let $U_{ij}^n=u(x_i,y_j,t_n),~F_{ij}^n=f(x_i,y_j,t_n),~(i,j)\in\bar{\omega},~0\leq n\leq N$, then Eq. \eqref{6.4} can be expressed as
\begin{align}
\int_1^2\omega(\alpha){}^C_0D_t^{\alpha}U_{ij}^{n-\frac{1}{2}}d\alpha
=K_1\frac{\partial^{\beta}U_{ij}^{n-\frac{1}{2}}}{\partial\mid x\mid^{\beta}}
+K_2\frac{\partial^{\gamma}U_{ij}^{n-\frac{1}{2}}}{\partial\mid y\mid^{\gamma}}
+F_{ij}^{n-\frac{1}{2}},~(i,j)\in\omega,~1\leq n\leq N.\label{6.5}
\end{align}

Using Lemma $\ref{lemma2.1}$, we get
\begin{equation}\label{6.6}
\int_1^2\omega(\alpha){}^C_0D_t^{\alpha}U_{ij}^{n-\frac{1}{2}}d\alpha
=\Delta\alpha\sum\limits_{l=0}^{2J}c_l\omega(\alpha_l){}^C_0D_t^{\alpha_l}U_{ij}^{n-\frac{1}{2}}
+\mathcal{O}(\Delta\alpha^2).
\end{equation}

According to the fully discrete difference scheme $(2.7)$ in \cite{Gao2017Two} and noticing the zero initial condition $\eqref{6.2}$, we have
\begin{align}
\Delta\alpha\sum\limits_{l=0}^{2J}c_l\omega(\alpha_l){}^C_0D_t^{\alpha_l}U_{ij}^{n-\frac{1}{2}}
%=&\Delta\alpha\sum\limits_{l=0}^{2J}c_l\omega(\alpha_l)_{-\infty}D_t^{\alpha_l}U_i^{n-\frac{1}{2}}\\
=\Delta\alpha\sum\limits_{l=0}^{2J}c_l\omega(\alpha_l)\frac{1}{\tau^{\gamma_l}}\sum\limits_{k=0}^{n-1}
\lambda_k^{(\gamma_l)}\delta_tU_{ij}^{n-k-\frac{1}{2}}+\mathcal{O}(\tau^2). \label{6.7}
\end{align}

Moreover, by the Lemma $\ref{lemma2.2}$ it is easy to know that
\begin{equation}
\frac{\partial^{\beta}U_{ij}^{n-\frac{1}{2}}}{\partial\mid x\mid^{\beta}}=
-h_1^{-\beta}\sum\limits_{k=i-M_1}^i\hat{g}_k^{(\beta)}U_{i-k,j}^{n-\frac{1}{2}}+\mathcal{O}(h_1^2). \label{6.8}
\end{equation}
\begin{equation}
\frac{\partial^{\gamma}U_{ij}^{n-\frac{1}{2}}}{\partial\mid y\mid^{\gamma}}=
-h_2^{-\gamma}\sum\limits_{k=j-M_2}^j\hat{g}_k^{(\gamma)}U_{i,j-k}^{n-\frac{1}{2}}+\mathcal{O}(h_2^2). \label{6.9}
\end{equation}

By substituting $\eqref{6.6}$-$\eqref{6.9}$ into $\eqref{6.5}$, we have
\begin{align}
\nonumber
&\Delta\alpha\sum\limits_{l=0}^{2J}c_l\omega(\alpha_l)\frac{1}{\tau^{\gamma_l}}\sum\limits_{k=0}^{n-1}
\lambda_k^{(\gamma_l)}\delta_tU_{ij}^{n-k-\frac{1}{2}}\\
=&-K_1h_1^{-\beta}\sum\limits_{k=i-M_1}^i\hat{g}_k^{(\beta)}U_{i-k,j}^{n-\frac{1}{2}}
-K_2h_2^{-\gamma}\sum\limits_{k=j-M_2}^j\hat{g}_k^{(\gamma)}U_{i,j-k}^{n-\frac{1}{2}}
+F_{ij}^{n-\frac{1}{2}}+p_{ij}^n,\label{6.10}
\end{align}
where $p_{ij}^n=\mathcal{O}(\Delta\alpha^2+\tau^2+h_1^2+h_2^2)$, $(i,j)\in\omega,~1\leq n\leq N$.

Omitting the small term $p_{ij}^n$ in \eqref{6.10}, replacing $U_{ij}^k$ with its numerical one $u_{ij}^k$ and noticing the initial and boundary value conditions \eqref{6.2}-\eqref{6.3}, we can construct the numerical scheme of $\eqref{6.1}$-$\eqref{6.3}$ as follows:
\begin{align}
\nonumber
&\Delta\alpha\sum\limits_{l=0}^{2J}c_l\omega(\alpha_l)\frac{1}{\tau^{\gamma_l}}\sum\limits_{k=0}^{n-1}
\lambda_k^{(\gamma_l)}\delta_tu_{ij}^{n-k-\frac{1}{2}}
=-K_1h_1^{-\beta}\sum\limits_{k=i-M_1}^i\hat{g}_k^{(\beta)}u_{i-k,j}^{n-\frac{1}{2}}
-K_2h_2^{-\gamma}\sum\limits_{k=j-M_2}^j\hat{g}_k^{(\gamma)}u_{i,j-k}^{n-\frac{1}{2}}\\
&\qquad\qquad\qquad\qquad\qquad\qquad\qquad\quad\qquad +F_{ij}^{n-\frac{1}{2}},\quad (i,j)\in\omega,~1\leq n\leq N,\label{6.11}\\
&u_{ij}^0=0,\quad (i,j)\in\omega,\label{6.12}\\
&u_{ij}^n=0,\quad (i,j)\in\partial\omega,~0\leq n\leq N.\label{6.13}
\end{align}

Similar to the Section \ref{section3}, the difference scheme \eqref{6.11}-\eqref{6.13} can be proved to be uniquely solvable, unconditionally stable and convergent with the order of $\mathcal{O}(\Delta\alpha^2+\tau^2+h_1^2+h_2^2)$.

Let
$$u^n=(u_{1,1}^n,\cdots,u_{M_1-1,1}^n,u_{1,2}^n,\cdots,u_{M_1-1,2}^n,u_{1,M_2-1}^n,\cdots,u_{M_1-1,M_2-1}^n)^T,$$
$$f^n=(F_{1,1}^n,\cdots,F_{M_1-1,1}^n,F_{1,2}^n,\cdots,F_{M_1-1,2}^n,F_{1,M_2-1}^n,\cdots,F_{M_1-1,M_2-1}^n)^T.$$

The following matrix form for difference scheme \eqref{6.11} can get
\begin{equation}\label{6.14}
 A_2u^n=p^{n-1},\quad n=1,2,\ldots,N,%\tag{2.13}
\end{equation}
where
 \begin{equation}\label{6.15}
 A_2=\mu_0I_1+I_2\otimes (K_1\nu_{\beta}G_\beta)+(K_2\nu_{\gamma}G_\gamma)\otimes I_3,
 \end{equation}
 and
 \begin{equation*}
 p^{n-1}=-[I_2\otimes (K_1\nu_{\beta}G_\beta)+(K_2\nu_{\gamma}G_\gamma)\otimes I_3] u^{n-1}
 +\sum\limits_{k=1}^{n-1}(\mu_{k-1}-\mu_k)u^{n-k}
 +\frac{\tau}{2}(f^n+f^{n-1}),
 \end{equation*}
in which $\otimes$ denotes the Kronecker product, $I_1$, $I_2$ and $I_3$ are identity matrices of order $(M_1-1)(M_2-1)$, $M_2-1$ and $M_1-1$, respectively, and the definitions of $\nu_{\gamma}$ and $G_\gamma$ are similar to $\nu_{\beta}$ and $G_\beta$, respectively.

\subsection{GL-PCG method with truncated preconditioner}\label{section6.2}
We remark that the PCG-based GSF method cannot be directly applied to handle the linear systems \eqref{6.14} since $A_2$ is a BTTB matrix and not a Toeplitz matrix. Therefore, in this subsection, we seek other effective algorithms to efficiently solve linear systems \eqref{6.14}.

Let $U^n=\rm{reshape}$$(u^n,M_1-1,M_2-1)$ and $F^n=\rm{reshape}$$(f^n,M_1-1,M_2-1)$.
It is easy to verify that the Kronecker equations \eqref{6.14} are equivalent to the following Sylvester matrix equations

\begin{equation}\label{6.16}
(\mu_0I_3+K_1\nu_{\beta}G_\beta){U}^n+{U}^n(K_2\nu_{\gamma}G_\gamma)={E}^{n-1},
\end{equation}
where
\begin{equation*}
{E}^{n-1}=-[(K_1\nu_{\beta}G_\beta){U}^{n-1}+{U}^{n-1}(K_2\nu_{\gamma}G_\gamma)]
+\sum\limits_{k=1}^{n-1}(\mu_{k-1}-\mu_k){U}^{n-k}
+\frac{\tau}{2}({F}^n+{F}^{n-1}).
\end{equation*}

\begin{proposition}
Let $\mathcal{L}$ be a linear operator as follows:
$$\mathcal{L}({U})=(\mu_0I_3+K_1\nu_{\beta}G_\beta){U}+{U}(K_2\nu_{\gamma}G_\gamma).$$
Then the linear operator $\mathcal{L}$ is SPD.
\label{proposition1}
\end{proposition}
\begin{proof}
The proof can be referred to the conclusion in \cite[Proposition 1]{Jafarbigloo2013Global}. %we omit it here.
\end{proof}

Since $\mathcal{L}$ is SPD, the GL-CG method \cite{Jafarbigloo2013Global} is a good choice for solving the Sylvester matrix equations \eqref{6.16}. However, the GL-CG method usually converges very slowly because the resulting matrix equations are often ill-conditioned.
Therefore, we will develop the preconditioned GL-CG (GL-PCG) method to accelerate the convergence rate.
More specifically, the GL-CG method is employed to solve its equivalent preconditioned linear matrix equation
$$\mathcal{\tilde{L}}^{-1}(\mathcal{L}({U}))=\mathcal{\tilde{L}}^{-1}({E}),$$
where $\mathcal{\tilde{L}}$ is a preconditioner operator. The algorithm of the GL-PCG method is given as follows.
\begin{algorithm}[H]
\caption{GL-PCG for $\mathcal{L}({U})={E}$ with a preconditioner operator $\mathcal{\tilde{L}}$}
\label{alg4}
\begin{algorithmic}[1]
\small
% 1
\STATE {Given initial ${U}_0$}
% 2
\STATE {Compute ${R}_0={E}-\mathcal{L}({U_0})$}
% 3
\STATE {${Z}_0=\mathcal{\tilde{L}}^{-1}({R}_0)$, ${P}_0={Z}_0$}
% 4
\STATE{For $j=0,1,2,\cdot\cdot\cdot$, until convergence Do}
% 5
\STATE{$\alpha_j=\frac{<{R}_j,{Z}_j>_F}{<\mathcal{L}({P}_j),{P}_j>_F}$}
% 6
\STATE {${U}_{j+1}={U}_j+\alpha_j {P}_j$}
% 7
\STATE {${R}_{j+1}={R}_j-\alpha_j\mathcal{L}({P}_j)$}
% 8
\STATE {${Z}_{j+1}=\mathcal{\tilde{L}}^{-1}({R}_{j+1}$)}
% 9
\STATE {$\beta_j=\frac{<{Z}_{j+1},{R}_{j+1}>_F}{<{Z}_j,{R}_j>_F}$}
% 10
\STATE {${P}_{j+1}={Z}_{j+1}+\beta_j{P}_j$}
% 11
\STATE {End Do}
\end{algorithmic}
\end{algorithm}

In Step 5 of the GL-GCG algorithm, the inner product $<\cdot,\cdot>_F$ is defined as
$$<{R}_j,{R}_j>_F=\rm{tr}({R}_j^T{R}_j)=\sum\limits_{i=1}^{M_2-1}\mathbf{r}_{j_i}^T\mathbf{r}_{j_i},$$
where $\mathbf{r}_{j_i}$ is the $i$-th column of ${R}_j$.

According to the well-known fact that a preconditioner with the same structure of the preconditioned matrix is optimal \cite{Serra1994Preconditioning}. We now design a truncated preconditioner $\mathcal{\tilde{L}}_l$ which has the same structure of $\mathcal{L}$ as follows:
\begin{equation}\label{6.17}
\mathcal{\tilde{L}}_l({Z})=[\mu_0I_3+K_1\nu_{\beta}T_l(G_\beta)]{Z}+{Z}[K_2\nu_{\gamma}T_l(G_\gamma)],
\end{equation}
where $T_l(G_\beta)$ is a $l$-truncation for matrix $G_\beta$ ($1\leq l\leq M_1-1$). More precisely,
\begin{equation*}
T_l(G_\beta)=
\begin{bmatrix}
{\hat{g}}_0^{(\beta)}&\cdots&{\hat{g}}_{-l}^{(\beta)}&0&\cdots&0\\
\vdots&{\hat{g}}_{0}^{(\beta)}&\ddots&{\hat{g}}_{-l}^{(\beta)}&\ddots&\vdots\\
{\hat{g}}_l^{(\beta)}&\ddots&{\hat{g}}_{0}^{(\beta)}&\ddots&\ddots&0\\
0&\ddots&\ddots&\ddots&\ddots&{\hat{g}}_{-l}^{(\beta)}\\
\vdots&\ddots&{\hat{g}}_{l}^{(\beta)}&\ddots&{\hat{g}}_{0}^{(\beta)}&\vdots\\
0&\cdots&0&{\hat{g}}_{l}^{(\beta)}&\cdots&{\hat{g}}_{0}^{(\beta)}
\end{bmatrix}.
\end{equation*}

The invertibility of $\mathcal{\tilde{L}}_l$ is guaranteed by the following lemma.
\begin{lemma}
For $1\leq l\leq \min\{M_1-1,M_2-1\}$, the preconditioner operator $\mathcal{\tilde{L}}_l$ is SPD.
\label{lemma6.1}
\end{lemma}
\begin{proof}

The truncated preconditioner $\mathcal{\tilde{L}}_l$ can be expressed in the following Kronecker form
\begin{equation*}
 {K}_l=\mu_0I_1+I_2\otimes K_1\nu_{\beta}T_l(G_\beta)+K_2\nu_{\gamma}T_l(G_\gamma)\otimes I_3.
 \end{equation*}
It is easy to prove that ${K}_l$ is SPD. From \cite[Proposition 2]{Jafarbigloo2013Global},
we get that $\mathcal{\tilde{L}}_l$ is also SPD.
\end{proof}

We investigate numerically the eigenvalue distributions of the proposed preconditioned matrix ${K}_l^{-1}A_2$ , where $l=\frac{M}{2}$ (see Figs. \ref{fig6}-\ref{fig7}). It is obvious that the spectrum of the preconditioned matrix is clustered around 1.

Two major computational loads in Algorithm \ref{alg4} are $\mathcal{\tilde{L}}_l^{-1}({R})$ and $\mathcal{L}({P})$. The $\mathcal{\tilde{L}}_l^{-1}({R})$ can be calculated by solving the Sylvester equation $\mathcal{\tilde{L}}_l({Z})={R}$.
Moreover, we observe that
\begin{equation*}
\mathcal{L}({P})=\mu_0{P}+(K_1\nu_{\beta}G_\beta){P}+{P}(K_2\nu_{\gamma}G_\gamma)
=\mu_0{P}+K_1\nu_{\beta}(G_\beta{P})+K_2\nu_{\gamma}(G_\gamma{P}^T)^T,
\end{equation*}
which implies that the computation of $\mathcal{L}({P})$ includes two Toeplitz matrix multiplications, $G_\beta{P}$ and $G_\gamma{P}^T$. Therefore, it can be obtained by FFT2 with only $\mathcal{O}(M_1M_2\log(M_1M_2))$.

%and the computation of $\mathbf{Z}=\mathcal{\tilde{L}}^{-1}\mathbf{R}$ is equivalent to solving the following Sylvester equation
%\begin{equation}\label{6.17}
%[\mu_0I_3+K_1\nu_{\beta}T_l(G_\beta)]\mathbf{Z}+\mathbf{Z}[K_2\nu_{\gamma}T_l(G_\gamma)]=\mathbf{R},
%\end{equation}
%fig1
\begin{figure}[!hbt]
  \centering
  \subfigure{
    \includegraphics[width = 6 cm]{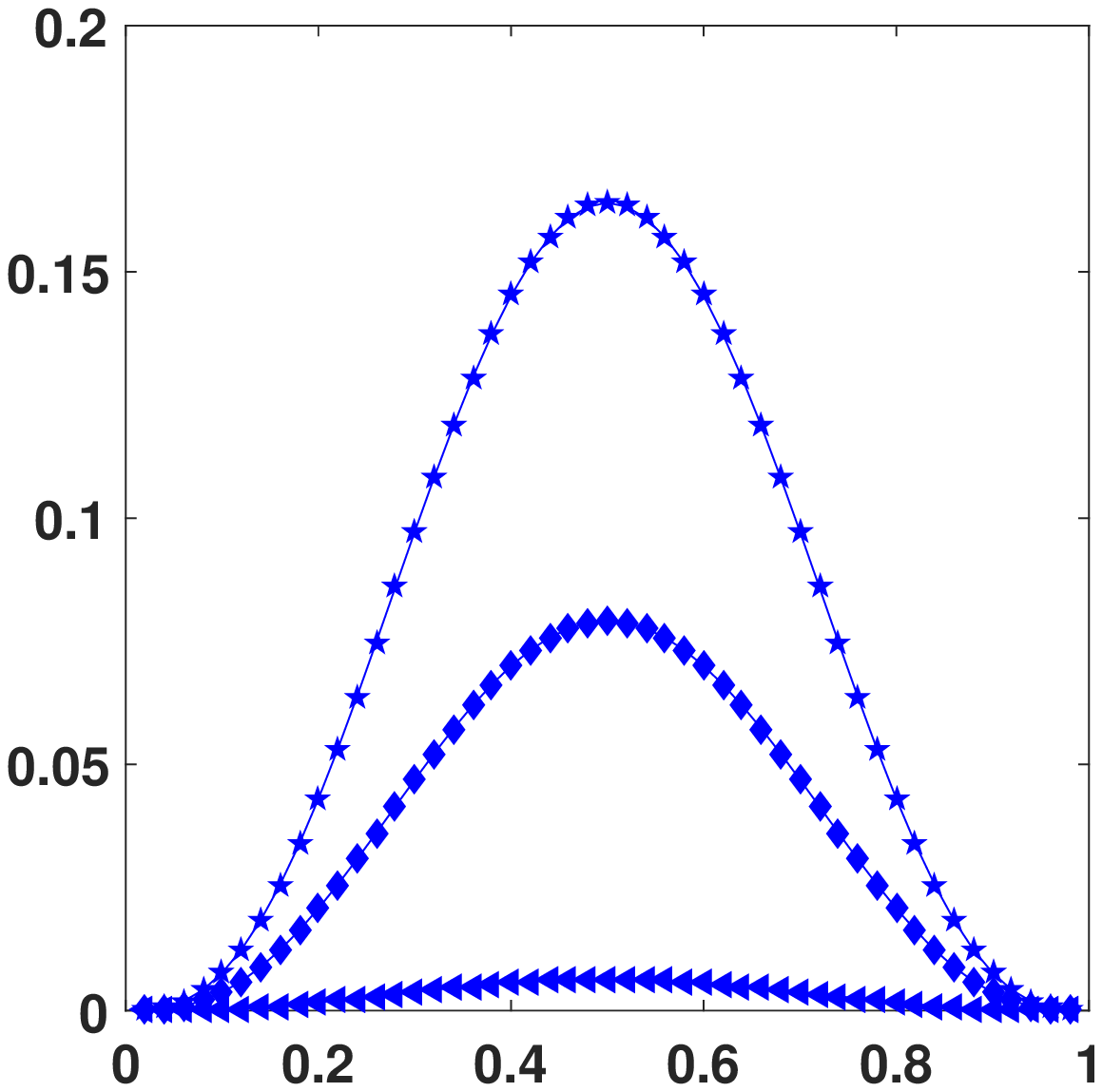}}
     \hspace{1 cm}
  \subfigure{
    \includegraphics[width = 6 cm]{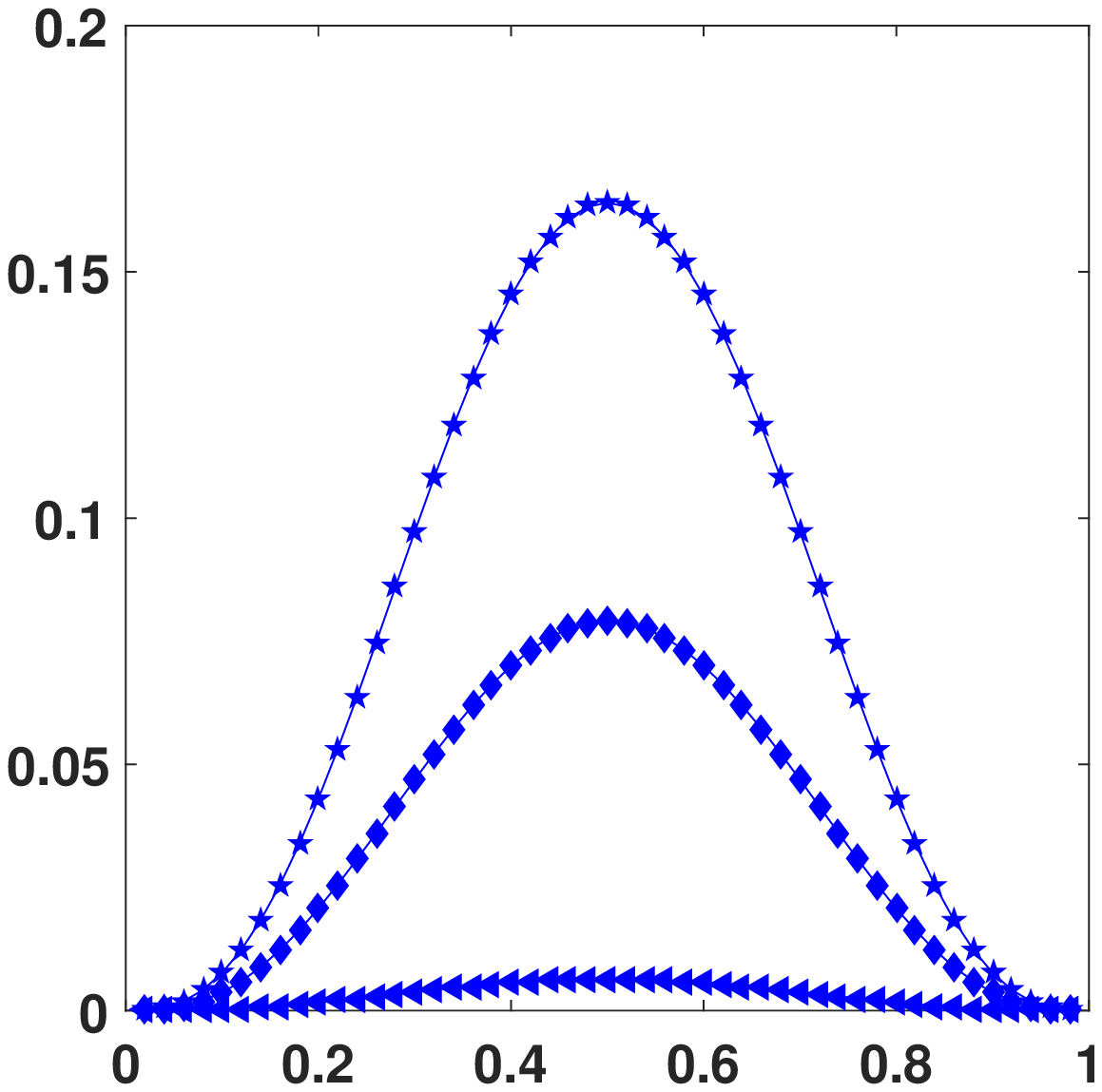}}\\
    (\textbf{a}) $\beta$ = 1.3 \hspace{5cm} (\textbf{b}) $\beta$ = 1.8 \\
  \caption{Exact (lines) and numerical (symbols) solutions of the scheme \eqref{2.11}-\eqref{2.13} for Example \ref{example1}: (\textbf{a}) $\beta$ = 1.3 at $T$ = 1.8 (stars), 1.5 (rhombus), 0.8 (triangles); (\textbf{b}) $\beta$ = 1.8 at $T$ = 1.8 (stars), 1.5 (rhombus), 0.8 (triangles), where $J=M=N=50$.}
  \label{fig1}
\end{figure}

\section{Numerical results}\label{section7}
Some numerical experiments are carried out in this section to verify the effectiveness of the proposed difference schemes and the performances of the fast solution techniques. All numerical experiments are implemented using MATLAB R2016a on a desktop with 16GB RAM, Inter (R) Core (TM) i7-8700K CPU @3.70GHz.

%table 1
\begin{table}[t]\small\tabcolsep=5.3pt
\begin{center}
\caption{{\small {The errors and space convergence orders of the difference scheme \eqref{2.11}-\eqref{2.13} for Example \ref{example1} with $T$ = 1.5, $J$ = 50, $N$ = 2000 and different $\beta$.}}}
\begin{tabular}{ccccccc}
\hline  & \multicolumn{2}{c}{$\beta=1.2$} & \multicolumn{2}
{c}{$\beta=1.5$} & \multicolumn{2}{c}{$\beta=1.8$} \\
[-2pt] \cmidrule(lr){2-3} \cmidrule(lr){4-5} \cmidrule(lr){6-7} \\ [-11pt]
 $M$ & $e(h,\tau,\Delta\alpha)$ & $rate_h$ & $e(h,\tau,\Delta\alpha)$ &$rate_h$ & $e(h,\tau,\Delta\alpha)$ & $rate_h$\\
\hline
16 & 1.444281e-04 &  -    & 2.611489e-04 &   -    & 3.877349e-04 &   -   \\
32 & 3.661426e-05 &1.9799 & 6.585700e-05 & 1.9875 & 9.793241e-05 & 1.9852\\
64 & 9.261868e-06 &1.9830 & 1.669595e-05 & 1.9798 & 2.469879e-05 & 1.9873 \\
128& 2.316352e-06 &1.9994 & 4.176866e-06 & 1.9990 & 6.183403e-06 & 1.9980 \\
256& 5.812125e-07 &1.9947 & 1.045642e-06 & 1.9980 & 1.547730e-06 & 1.9982 \\
%512& 1.353449e-07 &1.9990 & 2.404824e-07 & 1.9997 & 3.637954e-07 & 1.9946 \\
\hline
\end{tabular}
\label{tab1}
\end{center}
\end{table}
%table 2
\begin{table}[t]\small\tabcolsep=5.3pt
\begin{center}
\caption{{\small {The errors and time convergence orders of the difference scheme \eqref{2.11}-\eqref{2.13} for Example \ref{example1} with $T$ = 1.5, $J$ = 100, $M$ = 2500 and different $\beta$.}}}
\begin{tabular}{ccccccc}
\hline  & \multicolumn{2}{c}{$\beta=1.2$} & \multicolumn{2}
{c}{$\beta=1.5$} & \multicolumn{2}{c}{$\beta=1.8$} \\
[-2pt] \cmidrule(lr){2-3} \cmidrule(lr){4-5} \cmidrule(lr){6-7} \\ [-11pt]
 $N$  & $e(h,\tau,\Delta\alpha)$ & $rate_\tau$ & $e(h,\tau,\Delta\alpha)$ &$rate_\tau$ & $e(h,\tau,\Delta\alpha)$ & $rate_\tau$\\
\hline
16& 4.790434e-04 &   -  & 3.999822e-04&   -  & 3.172525e-04&   -   \\
32& 1.215018e-04 &1.9792& 1.013382e-04&1.9808& 8.009703e-05&1.9858 \\
64& 3.058415e-05 &1.9901& 2.549518e-05&1.9909& 2.011908e-05&1.9932 \\
128&7.668025e-06 &1.9959& 6.393357e-06&1.9956& 5.045536e-06&1.9955 \\
256&1.915933e-06 &2.0008& 1.600711e-06&1.9979& 1.267602e-06&1.9929 \\
\hline
\end{tabular}
\label{tab2}
\end{center}
\end{table}

%table 3
\begin{table}[t]\small\tabcolsep=5.3pt
\begin{center}
\caption{{\small {The errors and distributed-order convergence orders of the difference scheme \eqref{2.11}-\eqref{2.13} for Example \ref{example1} with $T$ = 1.5, $M$ = 4000, $N$ = 4000 and different $\beta$.}}}
\begin{tabular}{ccccccc}
\hline  & \multicolumn{2}{c}{$\beta=1.2$} & \multicolumn{2}
{c}{$\beta=1.5$} & \multicolumn{2}{c}{$\beta=1.8$} \\
[-2pt] \cmidrule(lr){2-3} \cmidrule(lr){4-5} \cmidrule(lr){6-7} \\ [-11pt]
 $J$  & $e(h,\tau,\Delta\alpha)$ & $rate_{\Delta\alpha}$ & $e(h,\tau,\Delta\alpha)$ &$rate_{\Delta\alpha}$ & $e(h,\tau,\Delta\alpha)$ & $rate_{\Delta\alpha}$\\
\hline
1 & 1.056413e-04 &   -  & 8.938958e-05&   -  & 7.216282e-05&   -   \\
2 & 2.703158e-05 &1.9665& 2.288894e-05&1.9655& 1.850067e-05&1.9637 \\
4 & 6.789749e-06 &1.9932& 5.748842e-06&1.9933& 4.646376e-06&1.9934 \\
8 & 1.692476e-06 &2.0042& 1.431730e-06&2.0055& 1.155441e-06&2.0077\\
16& 4.158579e-07 &2.0250& 3.504514e-07&2.0305& 2.810033e-07&2.0398 \\
\hline
\end{tabular}
\label{tab3}
\end{center}
\end{table}

 %table 4
\begin{table}[t]\small\tabcolsep=2.5pt
\begin{center}
\caption{{\small {CPU comparisons for solving Example \ref{example1} between the Cholesky, PCG, and PCG-based GSF methods with different $\beta$ and preconditioners, where $J$=50 and $T$=1.5.}}}
\begin{tabular}{cccccccccccc}
\\
\hline & &\multicolumn{1}{c}{$\rm{Chol}$} & \multicolumn{2}
{c}{$\rm{PCG(R)}$}& \multicolumn{2}{c}{$\rm{PCG(S)}$} & \multicolumn{2}
{c}{\rm{PCG-GSF(R)}} & \multicolumn{2}{c}{\rm{PCG-GSF(S)}} \\
[-2pt] \cmidrule(lr){3-3} \cmidrule(lr){4-5} \cmidrule(lr){6-7} \cmidrule(lr){8-9} \cmidrule(lr){10-11} \\ [-11pt]
 $\beta$  & $M=N$ & \rm{CPU} & \rm{CPU(Iter)} & \rm{Speed-up} & \rm{CPU(Iter)} & \rm{Speed-up} & \rm{CPU} & \rm{Speed-up} & \rm{CPU} & \rm{Speed-up}\\
\hline
    &$2^6$  &0.01  &0.02(4) & 0.50 &0.02(4) & 0.50 &\textbf{0.02} & \textbf{0.50} &0.02  &0.50 \\
    &$2^7$  &0.07  &0.06(4) & 1.17 &0.06(4) & 1.17 &\textbf{0.05} & \textbf{1.40} &0.05  &1.40 \\
    &$2^8$  &0.25  &0.21(4) & 1.19 &0.22(4) & 1.14 &\textbf{0.20} & \textbf{1.25} &0.20  &1.25 \\
1.2 &$2^9$  &1.19  &0.84(4) & 1.42 &0.84(4) & 1.42 &\textbf{0.73} & \textbf{1.63} &0.73  &1.63 \\
    &$2^{10}$&6.44 &1.87(3) & 3.44 &1.86(3) & 3.46 &\textbf{1.69} & \textbf{3.81} &1.69  &3.81 \\
    &$2^{11}$&50.15&10.08(3)& 4.98 &10.11(3)& 4.96 &\textbf{8.83} & \textbf{5.68} &8.84  &5.67 \\
  \\
    &$2^6$  &0.01  &0.02(5) & 0.50 &0.02(4) & 0.50 &\textbf{0.02} & \textbf{0.50} &0.02  &0.50 \\
    &$2^7$  &0.07  &0.06(5) & 1.17 &0.06(4) & 1.17 &\textbf{0.05} & \textbf{1.40} &0.05  &1.40 \\
    &$2^8$  &0.25  &0.22(4) & 1.14 &0.22(4) & 1.14 &\textbf{0.20} & \textbf{1.25} &0.20  &1.25 \\
1.5 &$2^9$  &1.20  &0.84(4) & 1.43 &0.84(4) & 1.43 &\textbf{0.73} & \textbf{1.64} &0.73  &1.64 \\
    &$2^{10}$&6.51 &1.98(4) & 3.29 &1.97(4) & 3.30 &\textbf{1.69} & \textbf{3.85} &1.69  &3.85 \\
    &$2^{11}$&50.12&10.78(4)& 4.65 &10.79(4)& 4.65 &\textbf{8.85} & \textbf{5.66} &8.85  &5.66 \\
 \\
    &$2^6$  &0.01  &0.02(5) & 0.50 &0.02(5) & 0.50 &\textbf{0.02} & \textbf{0.50} &0.02  &0.05 \\
    &$2^7$  &0.07  &0.06(5) & 1.17 &0.06(5) & 1.17 &\textbf{0.05} & \textbf{1.40} &0.05  &1.40 \\
    &$2^8$  &0.25  &0.22(5) & 1.14 &0.22(5) & 1.14 &\textbf{0.20} & \textbf{1.25} &0.20  &1.25 \\
1.9 &$2^9$  &1.17  &0.84(4) & 1.39 &0.84(4) & 1.39 &\textbf{0.73} & \textbf{1.60} &0.73  &1.60 \\
    &$2^{10}$&6.45 &1.98(4) & 3.26 &1.97(4) & 3.27 &\textbf{1.69} & \textbf{3.82} &1.69  &3.82 \\
    &$2^{11}$&49.95&10.76(4)& 4.64 &10.78(4)& 4.63 &\textbf{8.86} & \textbf{5.64} &8.86  &5.64 \\
\hline
\end{tabular}
\label{tab4}
\end{center}
\end{table}

\begin{example}\label{example1}
We first consider the 1D distributed-order and Riesz space fractional diffusion-wave problem:
$$\begin{cases}
\int_1^2\Gamma(5-\alpha){}^C_0D_t^{\alpha}u(x,t)d\alpha=\frac{\partial^{\beta}u(x,t)}{\partial\mid
x\mid^{\beta}}+f(x,t),\quad 0<x<1,~0<t\leq T,\\
u(x,0)=0,\quad u_t(x,0)=0,\quad 0<x<1,\\
 u(0,t)=0,\quad u(1,t)=0,\quad 0\leq t\leq T,
\end{cases}$$
with $$f(x,t)=f_0(x,t)-ct^4\left[f_1(x)-3f_2(x)+3f_3(x)-f_4(x)\right],$$
where $c=-\frac{1}{2\cos(\beta\pi/2)}$, and $f_0(x,t)=24x^3(1-x)^3(t^3-t^2)/\ln t,$
\begin{align*}
&f_1(x)={\Gamma(4)}/{\Gamma(4-\beta)}[x^{3-\beta}+(1-x)^{3-\beta}],
~f_2(x)={\Gamma(5)}/{\Gamma(5-\beta)}[x^{4-\beta}+(1-x)^{4-\beta}],\\
&f_3(x)={\Gamma(6)}/{\Gamma(6-\beta)}[x^{5-\beta}+(1-x)^{5-\beta}],
~f_4(x)={\Gamma(7)}/{\Gamma(7-\beta)}[x^{6-\beta}+(1-x)^{6-\beta}].
\end{align*}
The exact solution of the above problem is $u(x,t)=t^4x^3(1-x)^3$.
\end{example}

Gorenflo et al. \cite{Gorenflo2013Fundamental} proposed that the solution $u(x,t)$ of the distributed order time-fractional diffusion-wave problem can be interpreted as a probability density function of the spatial variable $x$ evolving in time $t$. If we set $\omega(\alpha)=\delta(\alpha-\alpha_0),~1<\alpha_0\leq2$, it reduces to the single order time-fractional diffusion-wave equation \cite{Schneider1989Fractional}. When $\omega(\alpha)=\delta(\alpha-2\alpha_0)+2\lambda\delta(\alpha-\alpha_0),~0<\alpha_0\leq1$, it becomes the time-fractional telegraph equation\cite{Orsingher2004Time}. The probability density function is expressed through the weight function \cite{Chechkin2002Retarding}. We chose $\omega(\alpha)=\Gamma(5-\alpha)$ in our experiment as in the references \cite{bu2017finite,Gao2017Two,gao2015some}, which is advantageous to the exact solution is deduced.

To verify the convergence order of the proposed scheme \eqref{2.11}-\eqref{2.13}, in Tables \ref{tab1}-\ref{tab3}, let $$e(h,\tau,\Delta\alpha)=\max\limits_{0\leq i\leq M\atop0\leq n\leq N}|u(x_i,t_n,\Delta\alpha)-u_i^n|,$$
where $u(x_i,t_n,\Delta\alpha)$ and $u_i^n$ are the exact and numerical solutions at step sizes $h$, $\tau$ and $\Delta\alpha$, respectively. Thus the convergence orders are defined as
$$rate_h=\log_2\frac{e(h,\tau,\Delta\alpha)}{e(h/2,\tau,\Delta\alpha)},\quad
rate_{\tau}=\log_2\frac{e(h,\tau,\Delta\alpha)}{e(h,\tau/2,\Delta\alpha)},\quad
rate_{\Delta\alpha}=\log_2\frac{e(h,\tau,\Delta\alpha)}{e(h,\tau,{\Delta\alpha}/2)}.$$

Tables \ref{tab4}-\ref{tab6} are to verify the efficiency of the proposed PCG-based GSF method described in Section \ref{section4}. The resultant linear systems \eqref{2.14} of the 1D case are solved by using the Cholesky method, the PCG method and the proposed PCG-based GSF method, respectively.
The stopping criterion of all iterative methods is
$${\|r^{(k)}\|_2}/{\|r^{(0)}\|_2} < 10^{-12},$$
where $r^{(k)}$ is the residual vector after $k$ iterations, and the initial guess in each time level is given as the zero vector.

Besides the modified R. Chan's preconditioner $R$, we also test the Strang's circulant preconditioner \cite{gu2015strang}:
\begin{equation*}
 S=\mu_0I+K\nu_{\beta}s(G_\beta),
 \end{equation*}
where $s(\cdot)$ is the Strang's circulant preconditioner for an arbitrary Toeplitz matrix. We can also prove that the preconditioner $S$ has similar theoretical  properties to the preconditioner $R$ in the same way.
%More precisely, the first column of the criculant matrix $s(G_\beta)$ is given by
% \begin{center}
%$\left(
%  \begin{array}{c}
%    \hat{g}_0^{(\beta)} \\
%   \hat{g}_1^{(\beta)} \\
%   \hat{g}_2^{(\beta)} \\
%    \vdots \\
%    \hat{g}_{\lfloor\frac{M}{2}\rfloor-1}^{(\beta)}\\
%    \hat{g}_{\lfloor\frac{M}{2}\rfloor+1-M}^{(\beta)}\\
%    \vdots \\
%    \hat{g}_{-2}^{(\beta)} \\ \\
%    \hat{g}_{-1}^{(\beta)} \\ \\
%  \end{array}
%\right)$.
% \end{center}
%fig2
\begin{figure}[!hbt]
  \centering
  \subfigure{
    \includegraphics[width = 6.7 cm]{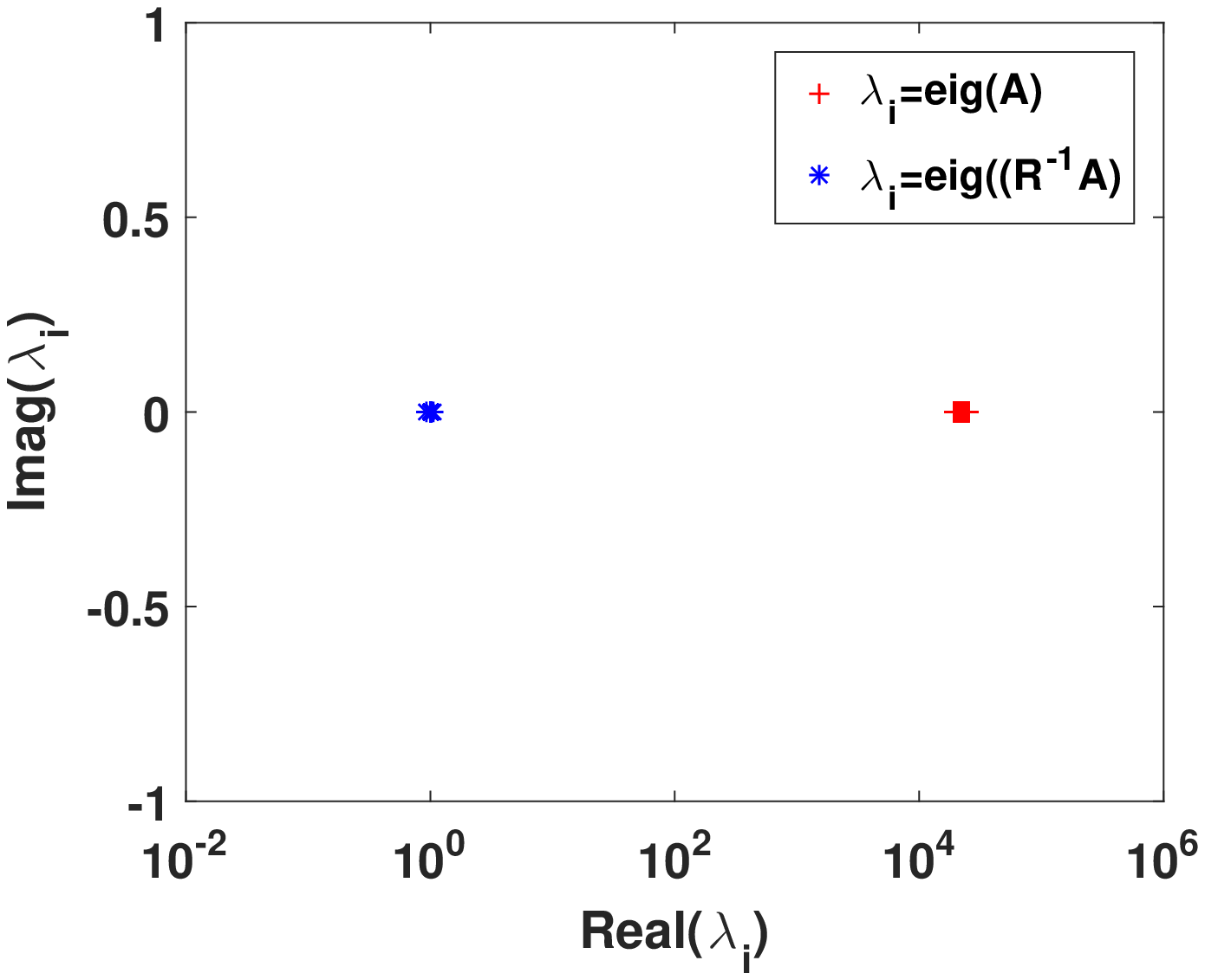}}
     \hspace{1 cm}
  \subfigure{
    \includegraphics[width = 6.7 cm]{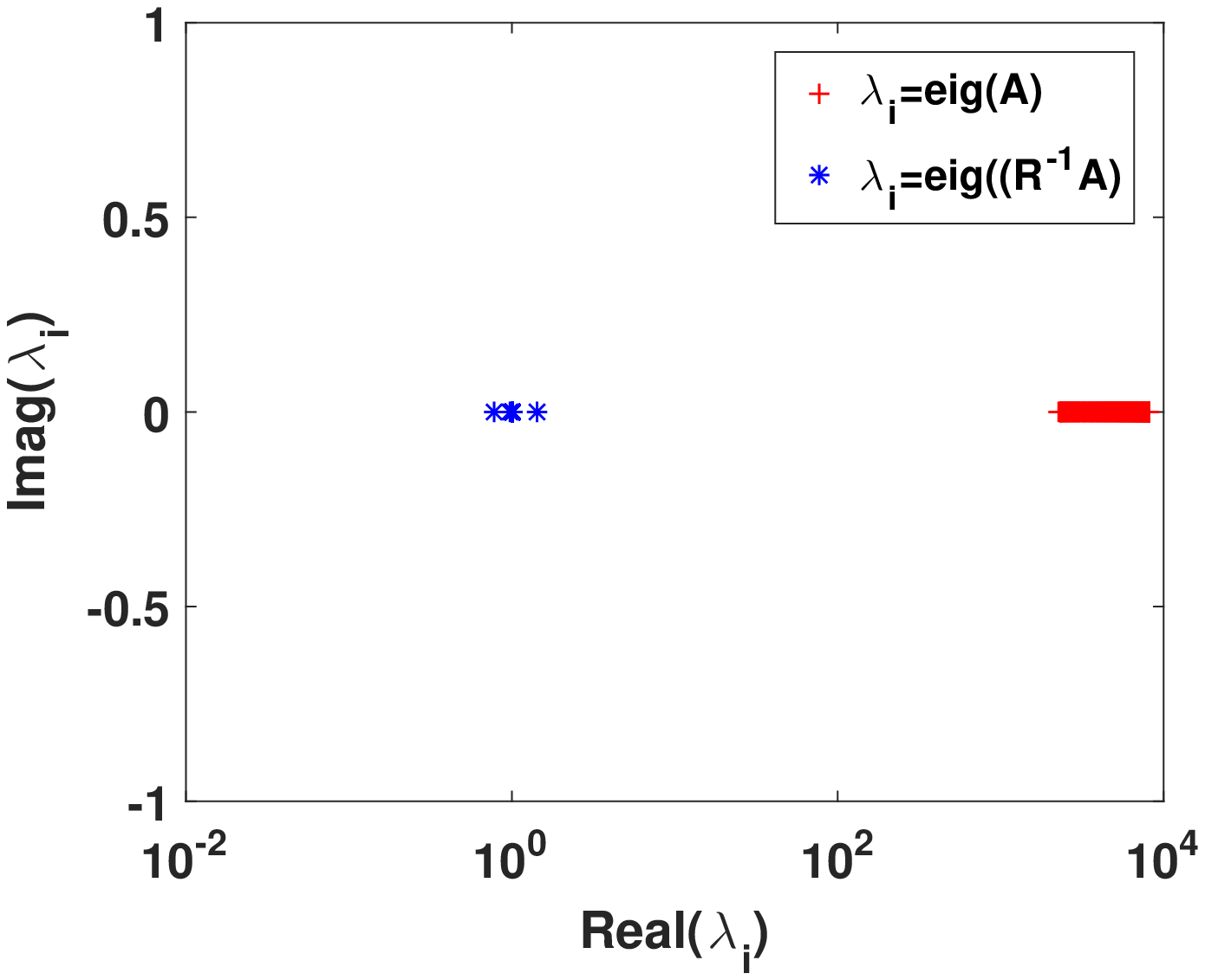}}\\
    (\textbf{a}) $T$=1.5 \hspace{5.8cm} (\textbf{b}) $T$ = 5 \\
  \caption{Spectrum of both original (red) and R. Chan's-based preconditioned (blue) matrices for Example \ref{example1} at (\textbf{a}) $T=1.5$ and (\textbf{b}) $T=5$, respectively, where $\beta$ = 1.5, $J$ = 50 and $M$ = $N$ = 256.}
  \label{fig2}
\end{figure}

In all tables, all times are the average of 10 runs of the algorithm. ``CPU" stands for the total CPU time in seconds to solve the whole linear systems. ``Chol" means that the Cholesky method is used to solve the discretized linear systems; ``PCG(R)" means that the PCG method with R. Chan's preconditioner is applied to solve the discretized linear systems; ``PCG-GSF(R)" means that the proposed PCG-based GSF method with R. Chan's preconditioner is applied to solve the discretized linear systems;
%the GSF method (Algorithm \ref{alg2}) is employed to solve the linear system in Step 3 of Algorithm \ref{alg1}, where the linear system $Al=e_1$ in Step 1 of Algorithm \ref{alg2} is solved by the PCG method with the R. Chan's-based preconditioner.
The meanings of ``PCG(S)" and ``PCG-GSF(S)" are similar to ``PCG(R)" and `PCG-GSF(R)", respectively.
In addition, ``Speed-up" denotes the run time speed-up compared to the Cholesky method, e.g.:
\begin{center}%\tiny\scriptsize
 Speed-up$_{\rm{PCG-GSF(R)}}=\frac{\rm{CPU_{Chol}}}{\rm{CPU_{PCG-GSF(R)}}}$.
\end{center}
``Iter" means the average number of iterations of the method for solving the discretized linear systems at all time levels, i.e.,
 \begin{equation*}
 \rm{Iter}=\frac{1}{N}\sum\limits_{n=1}^{N}Iter(n),
 \end{equation*}
where $\rm{Iter(n)}$ is the number of iterations at $n$-th time step.

%table 5
\begin{table}[t]\small\tabcolsep=2.5pt
\begin{center}
\caption{{\small {CPU comparisons for solving Example \ref{example1} between the Cholesky, PCG, and PCG-based GSF methods with different $\beta$ and preconditioners, where $J$=300 and $T$=1.5.}}}
\begin{tabular}{cccccccccccc}
\\
\hline & &\multicolumn{1}{c}{$\rm{Chol}$} & \multicolumn{2}
{c}{$\rm{PCG(R)}$}& \multicolumn{2}{c}{$\rm{PCG(S)}$} & \multicolumn{2}
{c}{\rm{PCG-GSF(R)}} & \multicolumn{2}{c}{\rm{PCG-GSF(S)}} \\
[-2pt] \cmidrule(lr){3-3} \cmidrule(lr){4-5} \cmidrule(lr){6-7} \cmidrule(lr){8-9} \cmidrule(lr){10-11} \\ [-11pt]
 $\beta$  & $M=N$ & \rm{CPU} & \rm{CPU(Iter)} & \rm{Speed-up} & \rm{CPU(Iter)} & \rm{Speed-up} & \rm{CPU} & \rm{Speed-up} & \rm{CPU} & \rm{Speed-up}\\
\hline
    &$2^6$  &0.01  &0.02(4) & 0.50 &0.02(4) & 0.50 &\textbf{0.02} & \textbf{0.50} &0.02 & 0.50\\
    &$2^7$  &0.07  &0.06(4) & 1.17 &0.06(4) & 1.17 &\textbf{0.05} & \textbf{1.40} &0.05 & 1.40\\
    &$2^8$  &0.25  &0.22(4) & 1.14 &0.22(4) & 1.14 &\textbf{0.20} & \textbf{1.25} &0.20 & 1.25\\
1.2 &$2^9$  &1.20  &0.84(4) & 1.43 &0.83(4) & 1.45 &\textbf{0.72} & \textbf{1.67} &0.72 & 1.67\\
    &$2^{10}$&6.49 &1.87(3) & 3.47 &1.84(3) & 3.53 &\textbf{1.67} & \textbf{3.89} &1.67 & 3.89\\
    &$2^{11}$&50.93&10.16(3)& 5.01 &10.03(3)& 5.08 &\textbf{8.69} & \textbf{5.86} &8.72 & 5.84\\
  \\
    &$2^6$  &0.01  &0.02(5) & 0.50 &0.02(4) & 0.50 &\textbf{0.02} & \textbf{0.50} &0.02 & 0.50\\
    &$2^7$  &0.06  &0.06(5) & 1.00 &0.06(4) & 1.00 &\textbf{0.05} & \textbf{1.20} &0.05 & 1.20\\
    &$2^8$  &0.25  &0.22(4) & 1.14 &0.22(4) & 1.14 &\textbf{0.20} & \textbf{1.25} &0.20 & 1.25\\
1.5 &$2^9$  &1.16  &0.84(4) & 1.38 &0.83(4) & 1.40 &\textbf{0.72} & \textbf{1.61} &0.72 & 1.61\\
    &$2^{10}$&6.40 &1.95(4) & 3.28 &1.95(4) & 3.28 &\textbf{1.67} & \textbf{3.83} &1.67 & 3.83\\
    &$2^{11}$&50.67&10.75(4)& 4.71 &10.70(4)& 4.74 &\textbf{8.71} & \textbf{5.82} &8.71 & 5.82\\
 \\
    &$2^6$  &0.01  &0.02(5) & 0.50 &0.02(5) & 0.50 &\textbf{0.02} & \textbf{0.50} &0.02 & 0.50\\
    &$2^7$  &0.07  &0.06(5) & 1.17 &0.06(5) & 1.17 &\textbf{0.05} & \textbf{1.40} &0.05 & 1.40\\
    &$2^8$  &0.24  &0.22(5) & 1.09 &0.22(5) & 1.09 &\textbf{0.20} & \textbf{1.20} &0.19 & 1.26\\
1.9 &$2^9$  &1.18  &0.83(4) & 1.42 &0.83(4) & 1.42 &\textbf{0.72} & \textbf{1.64} &0.72 & 1.64\\
    &$2^{10}$&6.45 &1.93(4) & 3.34 &1.95(4) & 3.31 &\textbf{1.67} & \textbf{3.86} &1.67 & 3.86\\
    &$2^{11}$&50.73&10.71(4)& 4.74 &10.71(4)& 4.74 &\textbf{8.69} & \textbf{5.84} &8.71 & 5.82\\
\hline
\end{tabular}
\label{tab5}
\end{center}
\end{table}

%table 6
\begin{table}[t]\small\tabcolsep=2.5pt
\begin{center}
\caption{{\small {CPU comparisons for solving Example \ref{example1} between the Cholesky, PCG, and PCG-based GSF methods with different $\beta$ and preconditioners, where $J$=50 and $T$=10.}}}
\begin{tabular}{cccccccccccc}
\\
\hline & &\multicolumn{1}{c}{$\rm{Chol}$} & \multicolumn{2}
{c}{$\rm{PCG(R)}$}& \multicolumn{2}{c}{$\rm{PCG(S)}$} & \multicolumn{2}
{c}{\rm{PCG-GSF(R)}} & \multicolumn{2}{c}{\rm{PCG-GSF(S)}} \\
[-2pt] \cmidrule(lr){3-3} \cmidrule(lr){4-5} \cmidrule(lr){6-7} \cmidrule(lr){8-9} \cmidrule(lr){10-11} \\ [-11pt]
 $\beta$  & $M=N$ & \rm{CPU} & \rm{CPU(Iter)} & \rm{Speed-up} & \rm{CPU(Iter)} & \rm{Speed-up} & \rm{CPU} & \rm{Speed-up} & \rm{CPU} & \rm{Speed-up}\\
\hline
    &$2^6$  &0.01  &0.02(5) & 0.50 &0.02(6) & 0.50 &\textbf{0.02} & \textbf{0.50} &0.02 & 0.50\\
    &$2^7$  &0.07  &0.06(5) & 1.17 &0.07(6) & 1.00 &\textbf{0.05} & \textbf{1.40} &0.05 & 1.40\\
    &$2^8$  &0.24  &0.23(5) & 1.04 &0.23(5) & 1.04 &\textbf{0.20} & \textbf{1.20} &0.20 & 1.20\\
1.2 &$2^9$  &1.19  &0.88(5) & 1.35 &0.88(5) & 1.35 &\textbf{0.73} & \textbf{1.63} &0.73 & 1.63\\
    &$2^{10}$&6.45 &2.08(5) & 3.10 &2.07(5) & 3.12 &\textbf{1.68} & \textbf{3.84} &1.69 & 3.82\\
    &$2^{11}$&53.21&11.43(5)& 4.66 &11.45(5)& 4.65 &\textbf{8.80} & \textbf{6.05} &8.81 & 6.04\\
  \\
    &$2^6$  &0.01  &0.02(6) & 0.50 &0.02(7) & 0.50 &\textbf{0.02} & \textbf{0.50} &0.02 & 0.50\\
    &$2^7$  &0.07  &0.07(6) & 1.00 &0.07(7) & 1.00 &\textbf{0.05} & \textbf{1.40} &0.05 & 1.40\\
    &$2^8$  &0.24  &0.24(6) & 1.00 &0.24(6) & 1.00 &\textbf{0.20} & \textbf{1.20} &0.20 & 1.20\\
1.5 &$2^9$  &1.19  &0.93(6) & 1.28 &0.88(5) & 1.35 &\textbf{0.73} & \textbf{1.63} &0.73 & 1.63\\
    &$2^{10}$&6.46 &2.07(5) & 3.12 &2.07(5) & 3.12 &\textbf{1.69} & \textbf{3.82} &1.69 & 3.82\\
    &$2^{11}$&51.31&11.47(5)& 4.47 &11.49(5)& 4.47 &\textbf{8.77} & \textbf{5.85} &8.78 & 5.84\\
 \\
    &$2^6$  &0.01  &0.02(5) & 0.50 &0.02(7) & 0.50 &\textbf{0.02} & \textbf{0.50} &0.02 & 0.50\\
    &$2^7$  &0.07  &0.07(6) & 1.00 &0.07(6) & 1.00 &\textbf{0.05} & \textbf{1.40} &0.05 & 1.40\\
    &$2^8$  &0.24  &0.24(6) & 1.00 &0.24(6) & 1.00 &\textbf{0.20} & \textbf{1.20} &0.20 & 1.20\\
1.9 &$2^9$  &1.21  &0.92(6) & 1.32 &0.88(5) & 1.38 &\textbf{0.73} & \textbf{1.66} &0.73 & 1.66\\
    &$2^{10}$&6.44 &2.08(5) & 3.10 &2.07(5) & 3.11 &\textbf{1.69} & \textbf{3.81} &1.69 & 3.81\\
    &$2^{11}$&51.14&11.48(5)& 4.45 &11.48(5)& 4.45 &\textbf{8.80} & \textbf{5.81} &8.81 & 5.80\\
\hline
\end{tabular}
\label{tab6}
\end{center}
\end{table}

%fig3
\begin{figure}[!hbt]
  \centering
  \subfigure{
    \includegraphics[width = 6.7 cm]{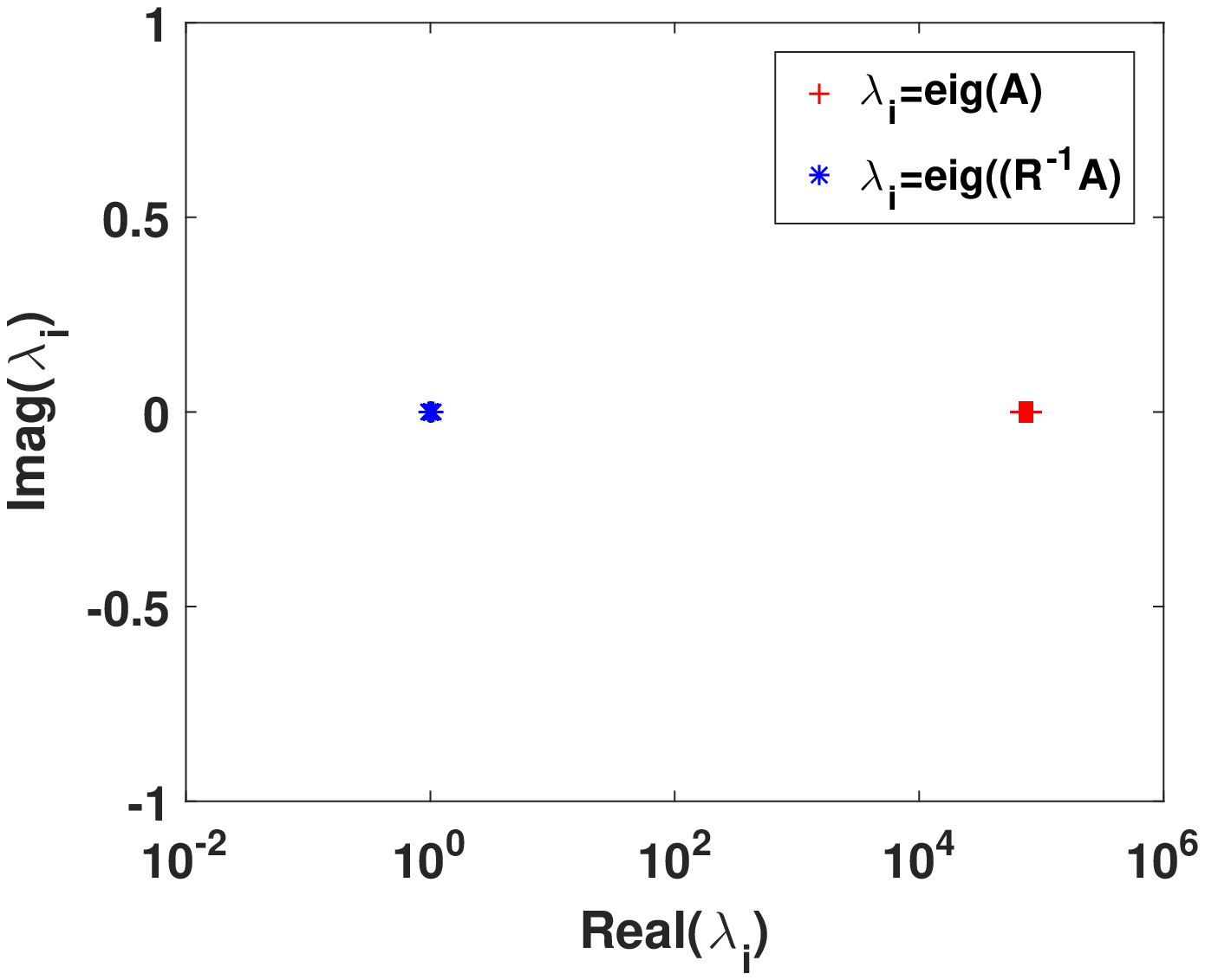}}
     \hspace{1 cm}
  \subfigure{
    \includegraphics[width = 6.7 cm]{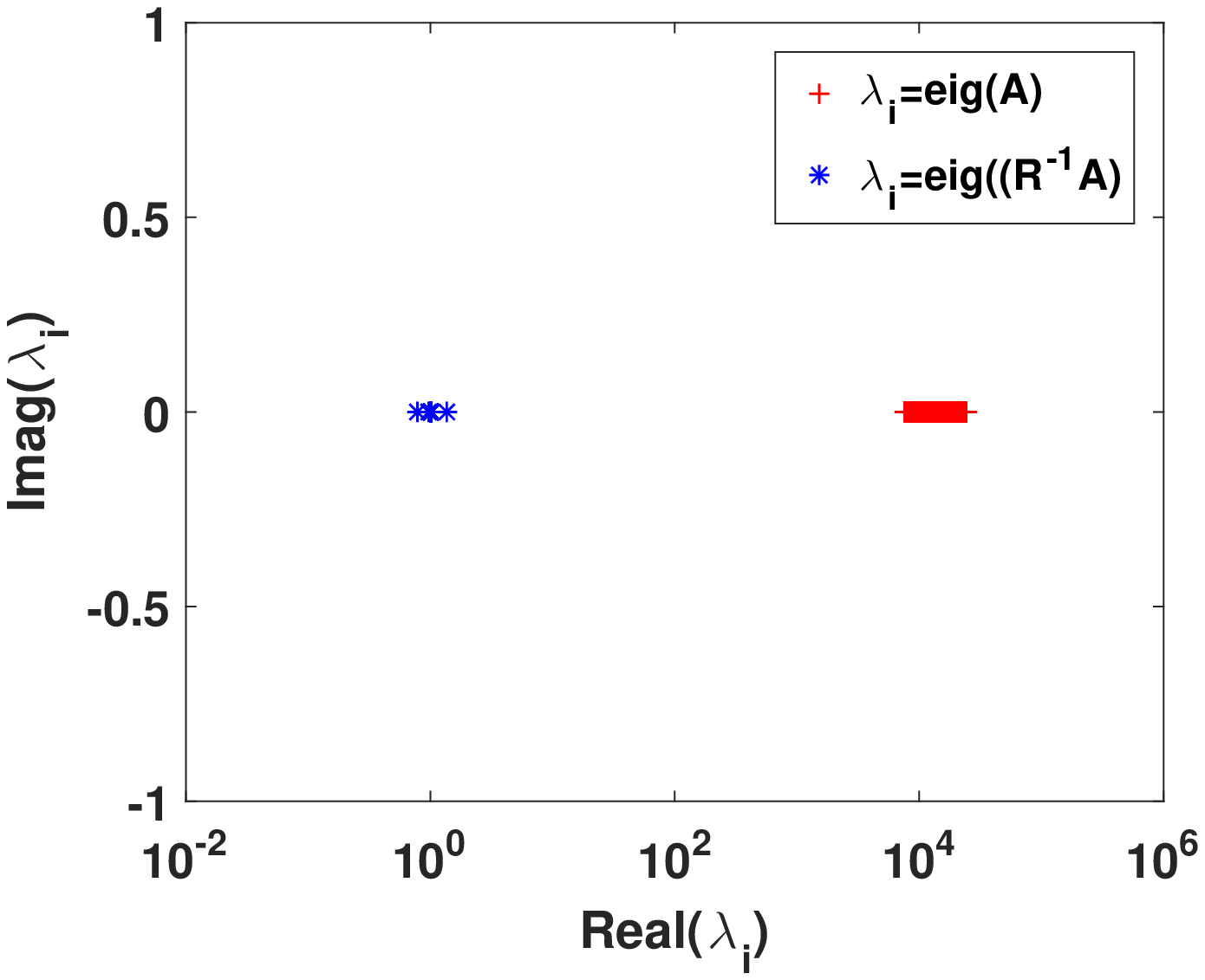}}\\
    (\textbf{a}) $T$=1.5 \hspace{5.8cm} (\textbf{b}) $T$ = 5 \\
  \caption{Spectrum of both original (red) and R. Chan's-based preconditioned (blue) matrices for Example \ref{example1} at (\textbf{a}) $T=1.5$ and (\textbf{b}) $T=5$, respectively, where $\beta$ = 1.5, $J$ = 50 and $M$ = $N$ = 512.}
  \label{fig3}
\end{figure}

The comparisons of the exact and numerical solutions of the numerical scheme \eqref{2.11}-\eqref{2.13} for Example \ref{example1} with different $\beta$ and $T$ are shown in Fig. \ref{fig1}. We can see that the numerical solutions are in good agreement with the exact solutions.

Tables \ref{tab1}-\ref{tab3} give the maximum errors and convergence orders of the numerical scheme \eqref{2.11}-\eqref{2.13} for solving Example \ref{example1} in space, time, and distributed order, respectively. From them, we can observe that the convergence order of the numerical scheme \eqref{2.11}-\eqref{2.13} is $\mathcal{O}(h^2+\tau^2+\Delta\alpha^2)$ as anticipated.

In Figs. \ref{fig2}-\ref{fig3}, the eigenvalues of both the original matrix $A$ and the
R. Chan's preconditioned matrix $R^{-1}A$ for Example \ref{example1} with different
$T$ are plotted. It can be seen that all the spectrum of the preconditioned matrix
$R^{-1}A$ are well separated away from 0, and the eigenvalues lie within a small interval around 1, except for few outliers, which is in agreement with the theoretical analysis. It validates the robustness and effectiveness of the proposed preconditioner in terms of spectrum clustering.

In Tables \ref{tab4}-\ref{tab6}, we compare the CPU time for solving Example \ref{example1} by the Cholesky method, the PCG method and the PCG-based GSF method with circulant preconditioners $R$ and $S$.
It shows that the CPU time of the PCG-based GSF method with $R$ and $S$ circulant preconditioners is much fewer than the Cholesky method. When $M=N=2^{11}$, the Speed-up of the PCG-based GSF methods is more than five times, although the parameters $\beta$, $J$, and $T$ take different values. It also shows that the CPU time by the PCG-based GSF methods is less than that by the PCG methods.
In addition, the iteration numbers of the PCG(R) method almost keep constant when $M$ is increasing rapidly, which verifies the superlinear convergence of the PCG method with R. Chan-based circulant preconditioner numerically.
We also observe that the performances of the PCG-GSF(R) and PCG-GSF(S) methods are almost the same in terms of the CPU time. The proposed fast solution algorithm in Section \ref{section4} is better than the other testing methods.%in terms of the CPU times.}
%in Section \ref{section4} takes much less CPU time elapsed as $M$ and $N$ become large.}

%table 7
\begin{table}[t]\small\tabcolsep=5.3pt
\begin{center}
\caption{{\small {The errors and space convergence orders of the difference scheme \eqref{6.11}-\eqref{6.13} for Example \ref{example2} with $T$ = 1.5, $J$ = 50, $N$ = 2000 and different $\beta$.}}}
\begin{tabular}{ccccccc}
\hline  & \multicolumn{2}{c}{$\beta=\gamma=1.2$} & \multicolumn{2}
{c}{$\beta=\gamma=1.5$} & \multicolumn{2}{c}{$\beta=\gamma=1.8$} \\
[-2pt] \cmidrule(lr){2-3} \cmidrule(lr){4-5} \cmidrule(lr){6-7} \\ [-11pt]
 $\widetilde{M}$ & $e({\widetilde{h}},\tau,\Delta\alpha)$ & ${\widetilde{rate}_h}$ & $e({\widetilde{h}},\tau,\Delta\alpha)$ &${\widetilde{rate}_h}$ & $e({\widetilde{h}},\tau,\Delta\alpha)$ & ${\widetilde{rate}_h}$\\
\hline
8  & 1.369354e-05 &  -    & 2.371348e-05 &   -    & 3.576549e-05 &   -    \\
16 & 3.405397e-06 &2.0076 & 5.857784e-06 & 2.0173 & 8.708526e-06 & 2.0381 \\
32 & 8.502827e-07 &2.0018 & 1.460064e-06 & 2.0043 & 2.162141e-06 & 2.0100 \\
64 & 2.124501e-07 &2.0008 & 3.646990e-07 & 2.0013 & 5.395161e-07 & 2.0027 \\
128& 5.302274e-08 &2.0024 & 9.107704e-08 & 2.0015 & 1.347362e-07 & 2.0015 \\
\hline
\end{tabular}
\label{tab7}
\end{center}
\end{table}
%table 8
\begin{table}[t]\small\tabcolsep=5.3pt
\begin{center}
\caption{{\small {The errors and time convergence orders of the difference scheme \eqref{6.11}-\eqref{6.13} for Example \ref{example2} with $T$ = 1.5, $J$ = 50, $\widetilde{M}$ = 800 and different $\beta$.}}}
\begin{tabular}{ccccccc}
\hline  & \multicolumn{2}{c}{$\beta=\gamma=1.2$} & \multicolumn{2}
{c}{$\beta=\gamma=1.5$} & \multicolumn{2}{c}{$\beta=\gamma=1.8$} \\
[-2pt] \cmidrule(lr){2-3} \cmidrule(lr){4-5} \cmidrule(lr){6-7} \\ [-11pt]
 $N$  & $e({\widetilde{h}},\tau,\Delta\alpha)$ & ${\widetilde{rate}_\tau}$ & $e({\widetilde{h}},\tau,\Delta\alpha)$ &${\widetilde{rate}_\tau}$ & $e({\widetilde{h}},\tau,\Delta\alpha)$ & ${\widetilde{rate}_\tau}$\\
\hline
4 & 7.332513e-05 &   -  & 5.407815e-05&   -  & 3.758769e-05&   -   \\
8 & 1.931968e-05 &1.9242& 1.406583e-05&1.9429& 9.608692e-06&1.9678 \\
16& 4.921064e-06 &1.9730& 3.570603e-06&1.9780& 2.411094e-06&1.9947 \\
32& 1.241275e-06 &1.9871& 9.001434e-07&1.9879& 6.055447e-07&1.9934 \\
64& 3.122634e-07 &1.9910& 2.272263e-07&1.9860& 1.538685e-07&1.9765 \\
\hline
\end{tabular}
\label{tab8}
\end{center}
\end{table}

%table 9
\begin{table}[t]\small\tabcolsep=5.3pt
\begin{center}
\caption{{\small {The errors and distributed-order convergence orders of the difference scheme \eqref{6.11}-\eqref{6.13} for Example \ref{example2} with $T$ = 1.5, $\widetilde{M}$ = 800, $N$ = 2000 and different $\beta$.}}}
\begin{tabular}{ccccccc}
\hline  & \multicolumn{2}{c}{$\beta=\gamma=1.2$} & \multicolumn{2}
{c}{$\beta=\gamma=1.5$} & \multicolumn{2}{c}{$\beta=\gamma=1.8$} \\
[-2pt] \cmidrule(lr){2-3} \cmidrule(lr){4-5} \cmidrule(lr){6-7} \\ [-11pt]
 $J$  & $e(\widetilde{h},\tau,\Delta\alpha)$ & $\widetilde{rate}_{\Delta\alpha}$ & $e(\widetilde{h},\tau,\Delta\alpha)$ &$\widetilde{rate}_{\Delta\alpha}$ & $e(\widetilde{h},\tau,\Delta\alpha)$ & $\widetilde{rate}_{\Delta\alpha}$\\
\hline
1 & 1.063403e-06 &   -  & 8.405192e-07&   -  & 6.396684e-07&   -  \\
2 & 2.730836e-07 &1.9613& 2.143924e-07&1.9710& 1.618385e-07&1.9828\\
4 & 6.753548e-08 &2.0156& 5.205615e-08&2.0421& 3.804022e-08&2.0890\\
8 & 1.565594e-08 &2.1089& 1.111329e-08&2.2278& 6.956504e-09&2.4511\\
\hline
\end{tabular}
\label{tab9}
\end{center}
\end{table}

\begin{example}\label{example2}
Consider the following 2D time distributed-order and Riesz space fractional diffusion-wave problem:
$$\begin{cases}
\int_1^2\Gamma(5-\alpha){}^C_0D_t^{\alpha}u(x,y,t)d\alpha=\frac{\partial^{\beta}u(x,y,t)}{\partial\mid
x\mid^{\beta}}+\frac{\partial^{\gamma}u(x,y,t)}{\partial\mid
y\mid^{\gamma}}+f(x,y,t),\quad (x,y)\in\Omega,~0<t\leq T,\\
u(x,y,0)=0,\quad u_t(x,y,0)=0,\quad (x,y)\in\Omega,\\
u(x,y,t)=0,\quad (x,y)\in\partial\Omega,\quad 0\leq t\leq T,
\end{cases}$$
with $\Omega=(0,1)\times(0,1)$, the exact solution is $u(x,t)=t^4x^3(1-x)^3y^3(1-y)^3$, and
\begin{align*}
f(x,y,t)=&f_0(x,y,t)-c_1t^4y^3(1-y)^3[f_1(x)-3f_2(x)+3f_3(x)-f_4(x)]\\
&-c_2t^4x^3(1-x)^3[g_1(y)-3g_2(y)+3g_3(y)-g_4(y)],
\end{align*}
where
$c_1=-\frac{1}{2\cos(\beta\pi/2)},~ c_2=-\frac{1}{2\cos(\gamma\pi/2)},$
 and $f_0(x,y,t)=24x^3(1-x)^3y^3(1-y)^3(t^3-t^2)/\ln t$,
\begin{align*}
%&f_0(x,y,t)=24x^3(1-x)^3y^3(1-y)^3(t^3-t^2)/\ln t,\\
&f_1(x)={\Gamma(4)}/{\Gamma(4-\beta)}[x^{3-\beta}+(1-x)^{3-\beta}],
~f_2(x)={\Gamma(5)}/{\Gamma(5-\beta)}[x^{4-\beta}+(1-x)^{4-\beta}],\\
&f_3(x)={\Gamma(6)}/{\Gamma(6-\beta)}[x^{5-\beta}+(1-x)^{5-\beta}],
~f_4(x)={\Gamma(7)}/{\Gamma(7-\beta)}[x^{6-\beta}+(1-x)^{6-\beta}],\\
&g_1(y)={\Gamma(4)}/{\Gamma(4-\gamma)}[y^{3-\gamma}+(1-y)^{3-\gamma}],
~g_2(y)={\Gamma(5)}/{\Gamma(5-\gamma)}[y^{4-\gamma}+(1-y)^{4-\gamma}],\\
&g_3(y)={\Gamma(6)}/{\Gamma(6-\gamma)}[y^{5-\gamma}+(1-y)^{5-\gamma}],
~g_4(y)={\Gamma(7)}/{\Gamma(7-\gamma)}[y^{6-\gamma}+(1-y)^{6-\gamma}].
\end{align*}
\end{example}

%fig4
\begin{figure}[!hbt]
  \centering
  \subfigure{
    \includegraphics[width = 6 cm]{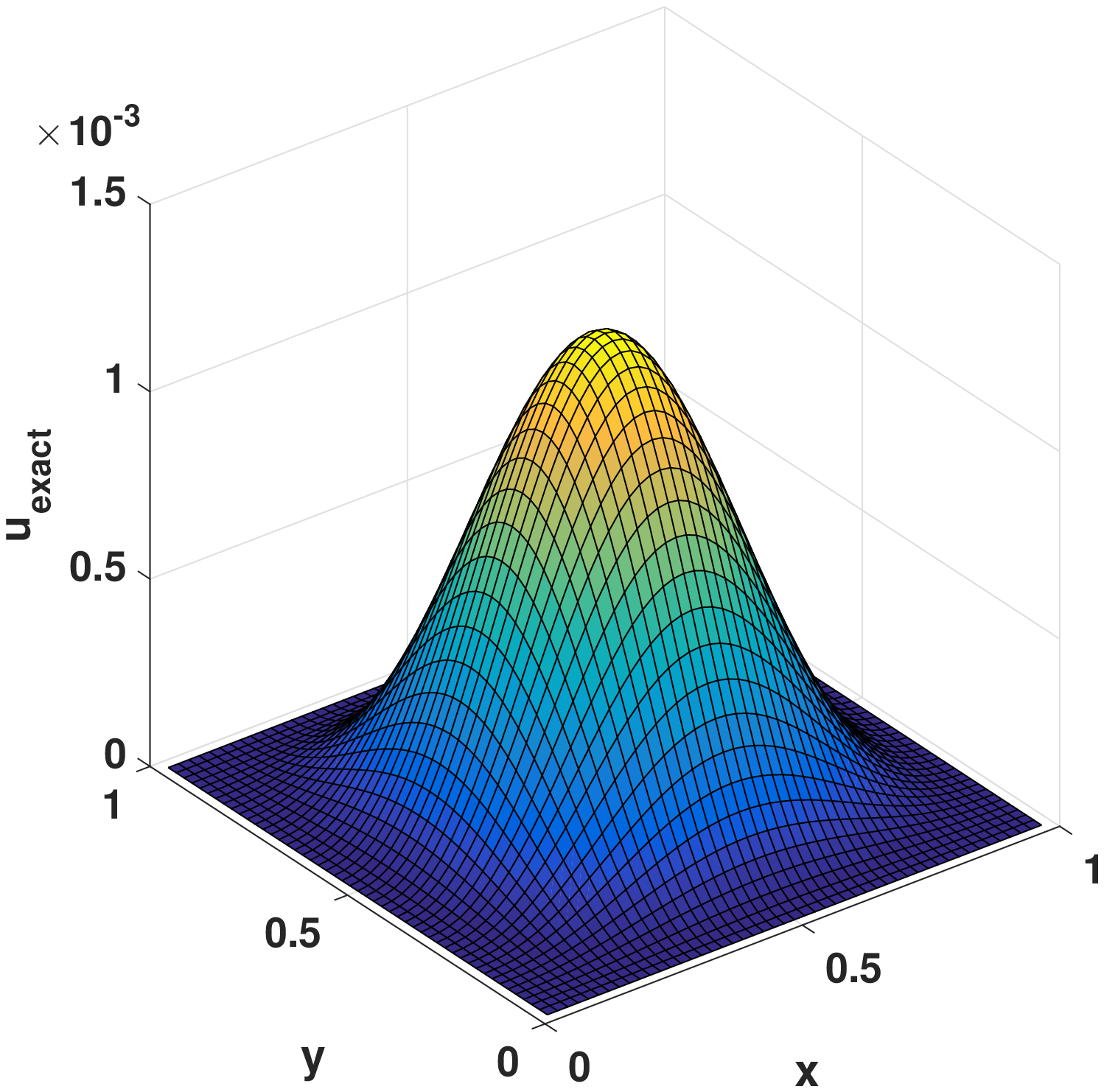}}
     \hspace{1 cm}
  \subfigure{
    \includegraphics[width = 6 cm]{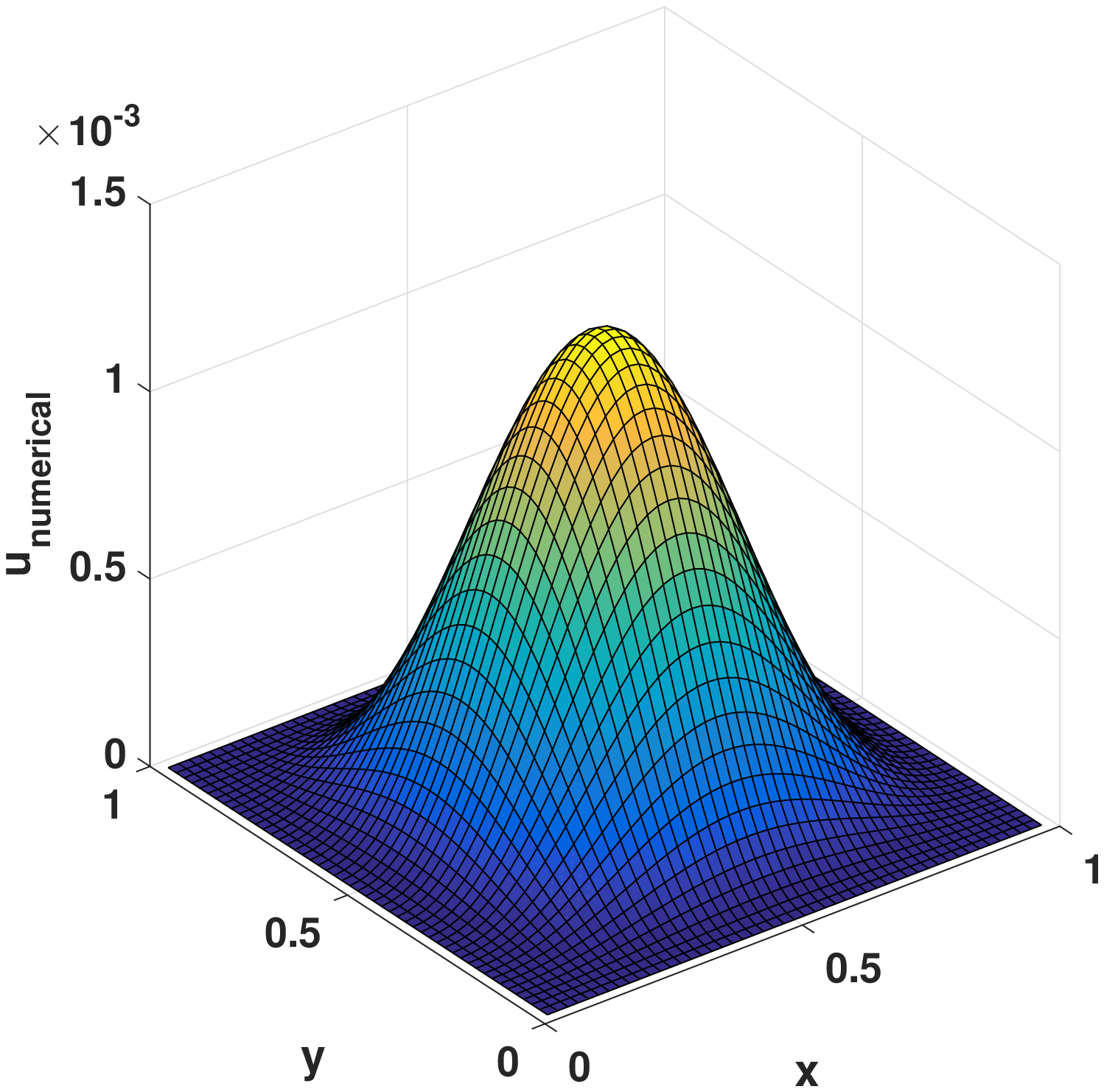}}\\
    %(\textbf{a}) Exact solution \hspace{3cm} (\textbf{b}) Numerical solution \\
  \caption{A comparison of the exact solution and the numerical solution of the scheme \eqref{6.11}-\eqref{6.13} for Example \ref{example2} at $\beta$ = $\gamma$ = 1.8, $T=1.5$ with $J$ = 50, $M_1=M_2$ =50, and $N$ = 10.}
  \label{fig4}
\end{figure}

%fig5
\begin{figure}[!hbt]
  \centering
  \subfigure{
    \includegraphics[width = 6 cm]{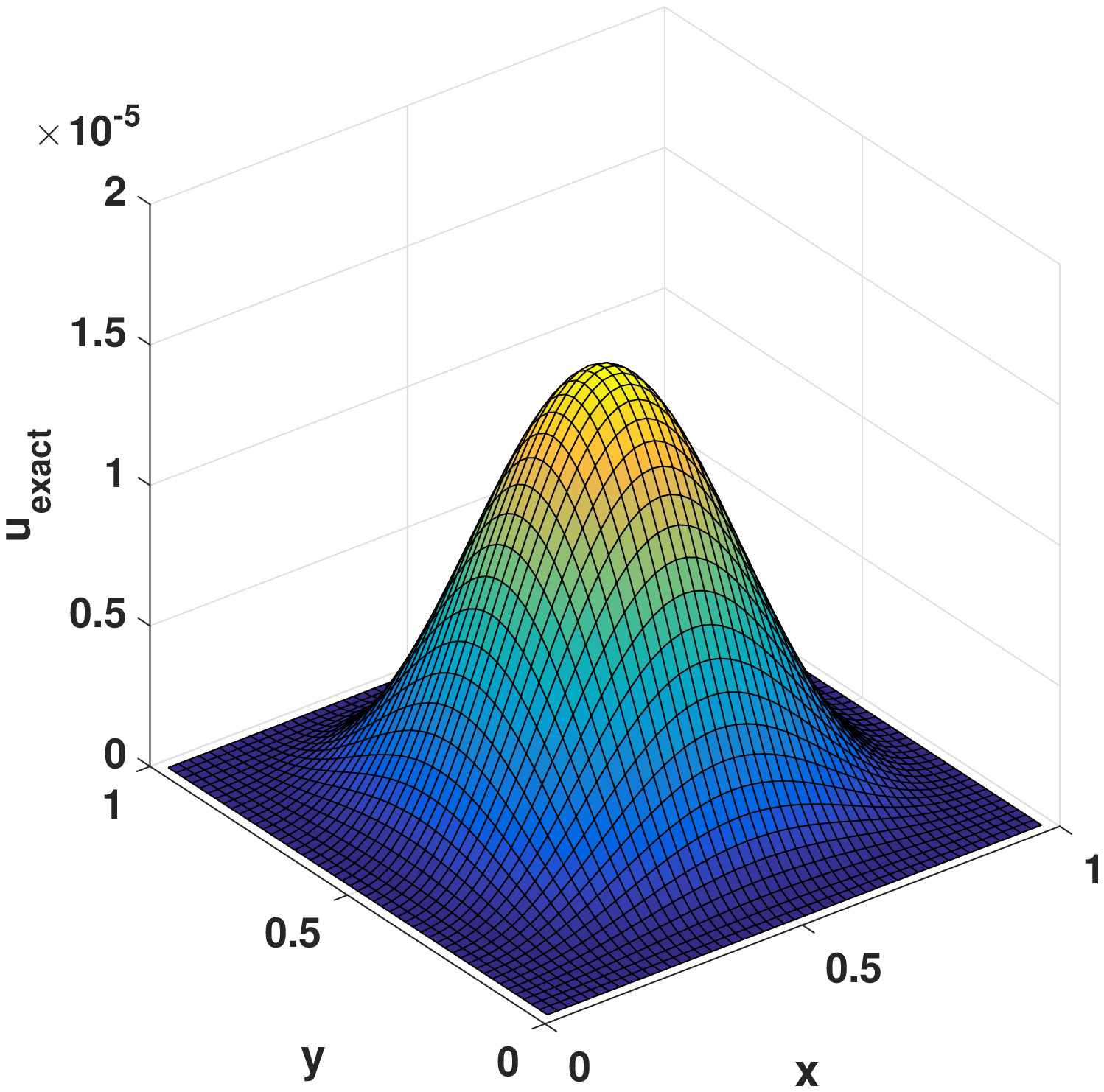}}
     \hspace{1 cm}
  \subfigure{
    \includegraphics[width = 6 cm]{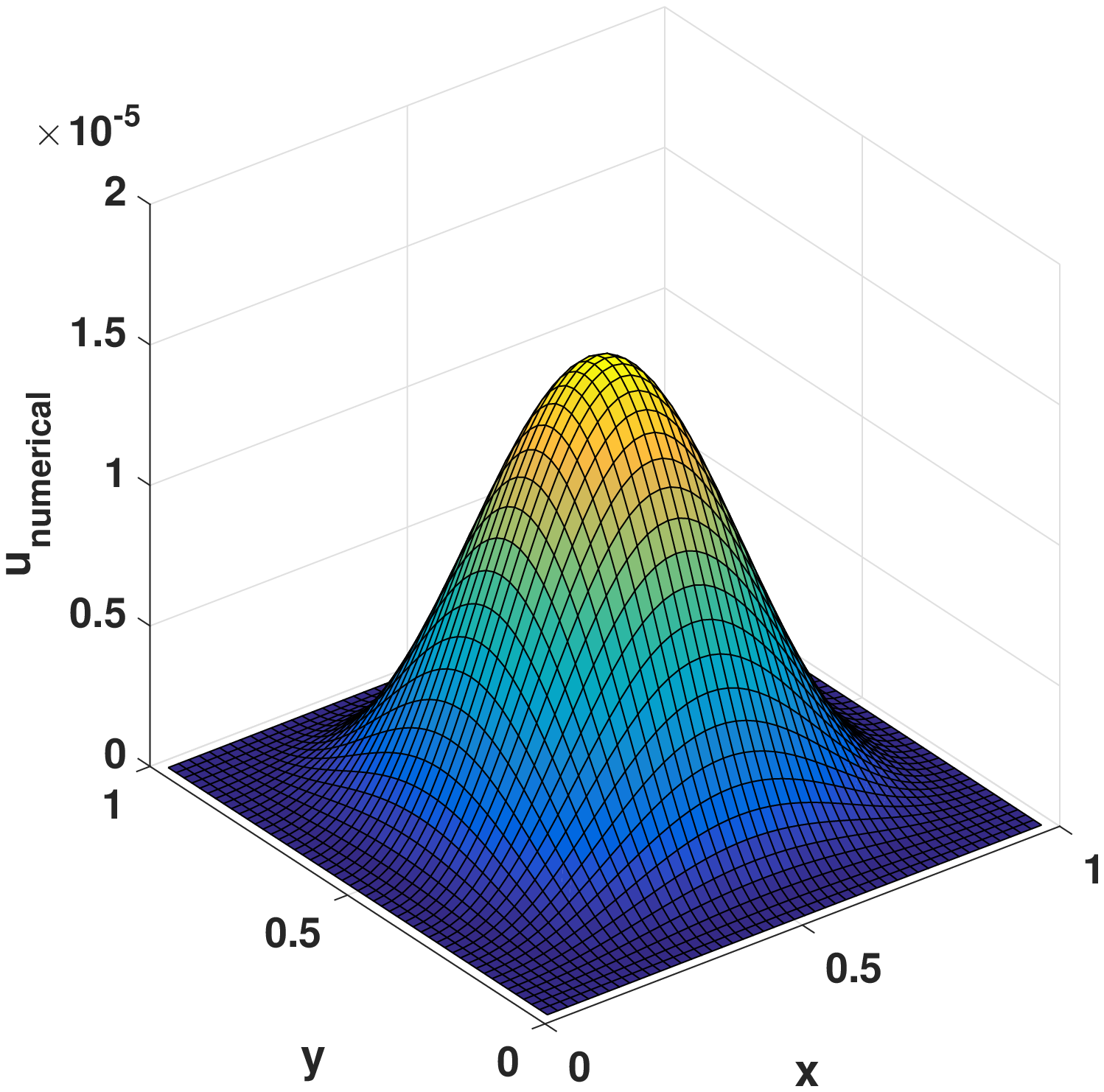}}\\
    %(\textbf{a}) Exact solution \hspace{3cm} (\textbf{b}) Numerical solution \\
  \caption{A comparison of the exact solution and the numerical solution of the scheme \eqref{6.11}-\eqref{6.13} for Example \ref{example2} at $\beta$ = $\gamma$ = 1.3, $T=0.5$ with $J$ = 50, $M_1=M_2$ =50, and $N$ = 10.}
  \label{fig5}
\end{figure}

%table 10
\begin{table}[t]\small\tabcolsep=9.0pt
\begin{center}
\caption{{\small {Comparisons for solving Example \ref{example2} between the Cholesky, PCG, and GL-PCG methods with different $\beta$ and preconditioners, where $l=\frac{M}{2}$, $J$ = 50 and $T$ = 1.5.}}}
\begin{tabular}{cccccccccc}
\\
\hline & & &\multicolumn{1}{c}{$\rm{Chol}$} & \multicolumn{2}
{c}{$\rm{PCG(R_2)}$}& \multicolumn{2}{c}{GL-PCG($R_2$)} & \multicolumn{2}
{c}{GL-PCG($\mathcal{\tilde{L}}_l$)} \\
[-2pt] \cmidrule(lr){4-4} \cmidrule(lr){5-6} \cmidrule(lr){7-8} \cmidrule(lr){9-10} \\ [-11pt]
 $\beta$=$\gamma$  & $\widetilde{M}$ & $N$ & $\rm{CPU}$ & $\rm{CPU}$ & $\rm{Iter}$ & $\rm{CPU}$ & $\rm{Iter}$ & $\rm{CPU}$ & $\rm{Iter}$ \\
\hline
    &$2^3$ &$2^3$  &0.00  &0.01 &5.8  &0.01 &5.6  &\textbf{0.01} &\textbf{4.0} \\
    &$2^4$ &$2^4$  &0.03  &0.03 &5.0  &0.03 &5.0  &\textbf{0.03} &\textbf{3.0} \\
1.2 &$2^5$ &$2^5$  &0.20  &0.22 &5.0  &0.23 &5.0  &\textbf{0.22} &\textbf{3.0} \\
    &$2^6$ &$2^6$  &2.94  &1.19 &4.0  &1.25 &4.0  &\textbf{1.33} &\textbf{3.0} \\
    &$2^7$ &$2^7$  &66.33 &9.16 &3.0  &8.57 &3.0  &\textbf{9.02} &\textbf{2.0} \\
  \\
    &$2^3$ &$2^3$  &0.00  &0.01 &6.0  &0.01 &6.0  &\textbf{0.01} &\textbf{4.0} \\
    &$2^4$ &$2^4$  &0.03  &0.03 &6.0  &0.03 &6.0  &\textbf{0.03} &\textbf{3.0} \\
1.5 &$2^5$ &$2^5$  &0.20  &0.23 &6.0  &0.23 &5.0  &\textbf{0.22} &\textbf{3.0} \\
    &$2^6$ &$2^6$  &2.91  &1.26 &5.0  &1.31 &5.0  &\textbf{1.34} &\textbf{3.0} \\
    &$2^7$ &$2^7$  &65.35 &9.73 &4.0  &9.04 &4.0  &\textbf{9.03} &\textbf{2.0} \\
 \\
    &$2^3$ &$2^3$  &0.00  &0.01 &7.0  &0.01 &7.0  &\textbf{0.01} &\textbf{3.0}  \\
    &$2^4$ &$2^4$  &0.03  &0.04 &7.0  &0.03 &7.0  &\textbf{0.03} &\textbf{3.0} \\
1.9 &$2^5$ &$2^5$  &0.21  &0.24 &7.0  &0.26 &7.0  &\textbf{0.22} &\textbf{3.0} \\
    &$2^6$ &$2^6$  &2.91  &1.30 &6.0  &1.35 &6.0  &\textbf{1.24} &\textbf{2.0} \\
    &$2^7$ &$2^7$  &65.39 &10.57&5.0  &9.78 &5.6  &\textbf{9.01} &\textbf{2.0} \\
\hline
\end{tabular}
\label{tab10}
\end{center}
\end{table}

%fig6
\begin{figure}[!hbt]
  \centering
  \subfigure{
    \includegraphics[width = 6.7 cm]{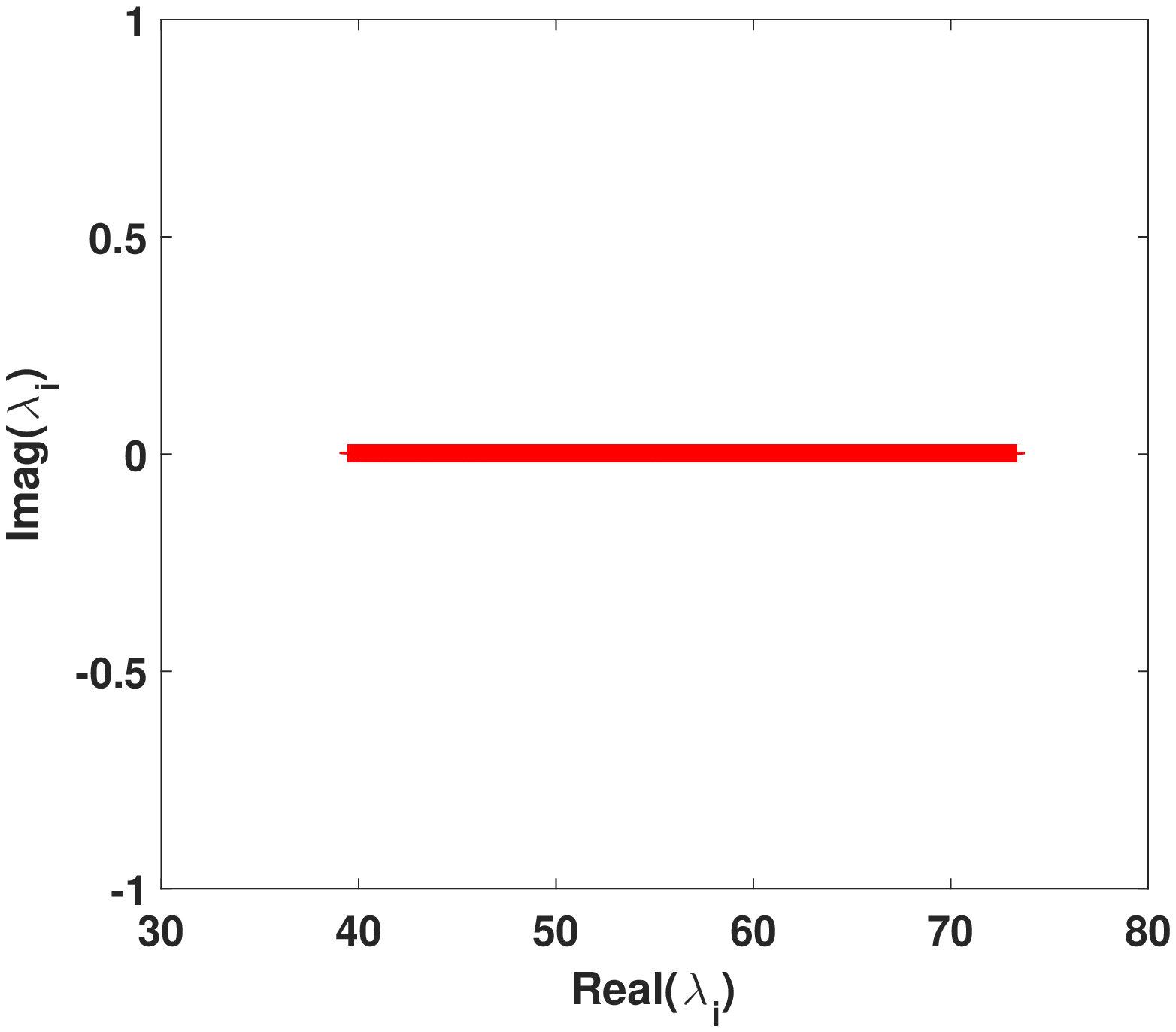}}
     \hspace{1 cm}
  \subfigure{
    \includegraphics[width = 6.7 cm]{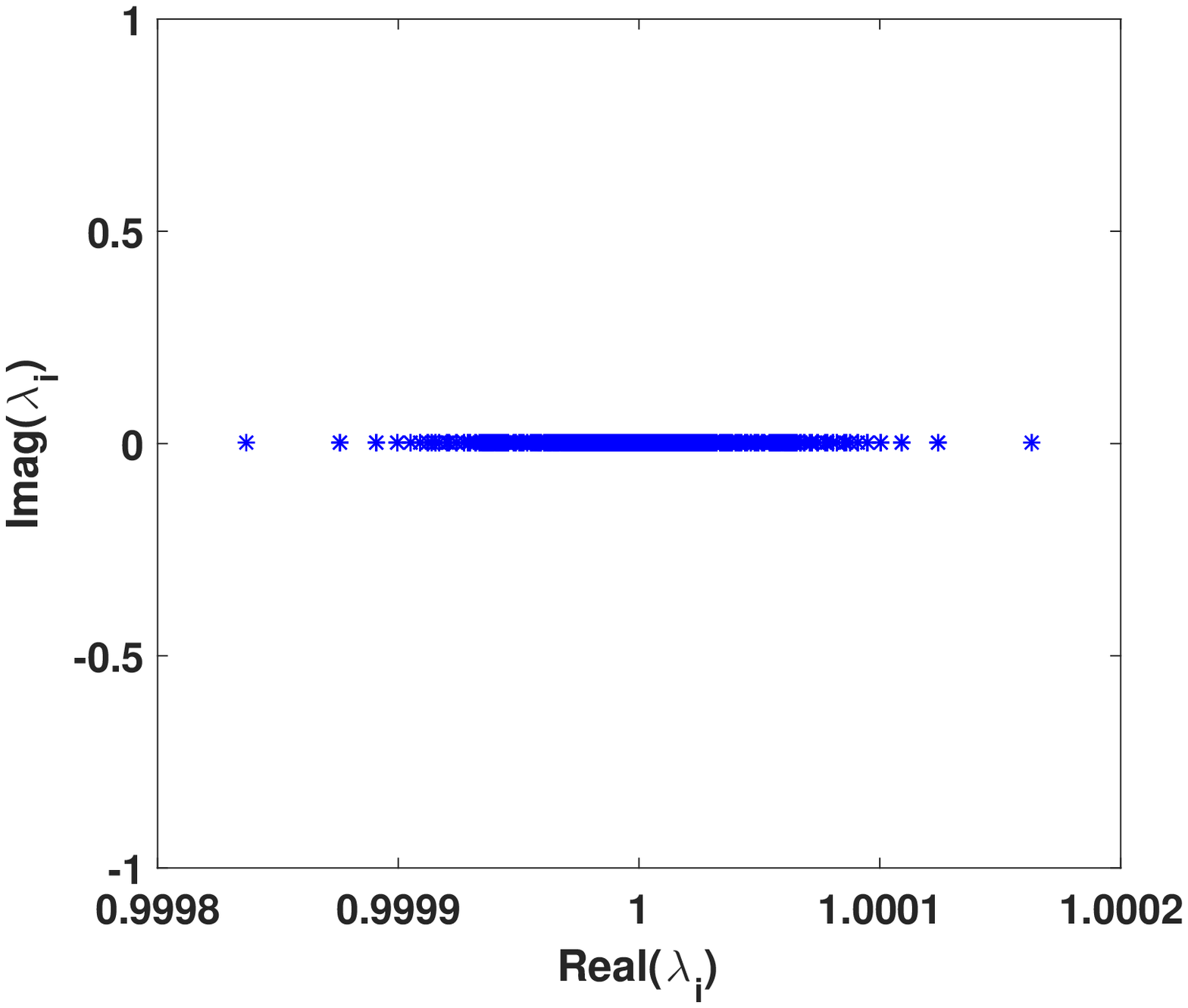}}\\
    (\textbf{a}) Spectrum of $A_2$ \hspace{4.6cm} (\textbf{b}) Spectrum of ${K}_l^{-1}A_2$ \\
  \caption{Spectrum of both original (\textbf{a}) and preconditioned (\textbf{b}) matrices for Example \ref{example2}, where $T$ = 1.5, $J$ = 50, $\beta$ = $\gamma$ = 1.5 and $\tilde{M}$ = $N$ = 64.}
  \label{fig6}
\end{figure}
%fig7
\begin{figure}[!hbt]
  \centering
  \subfigure{
    \includegraphics[width = 6.7 cm]{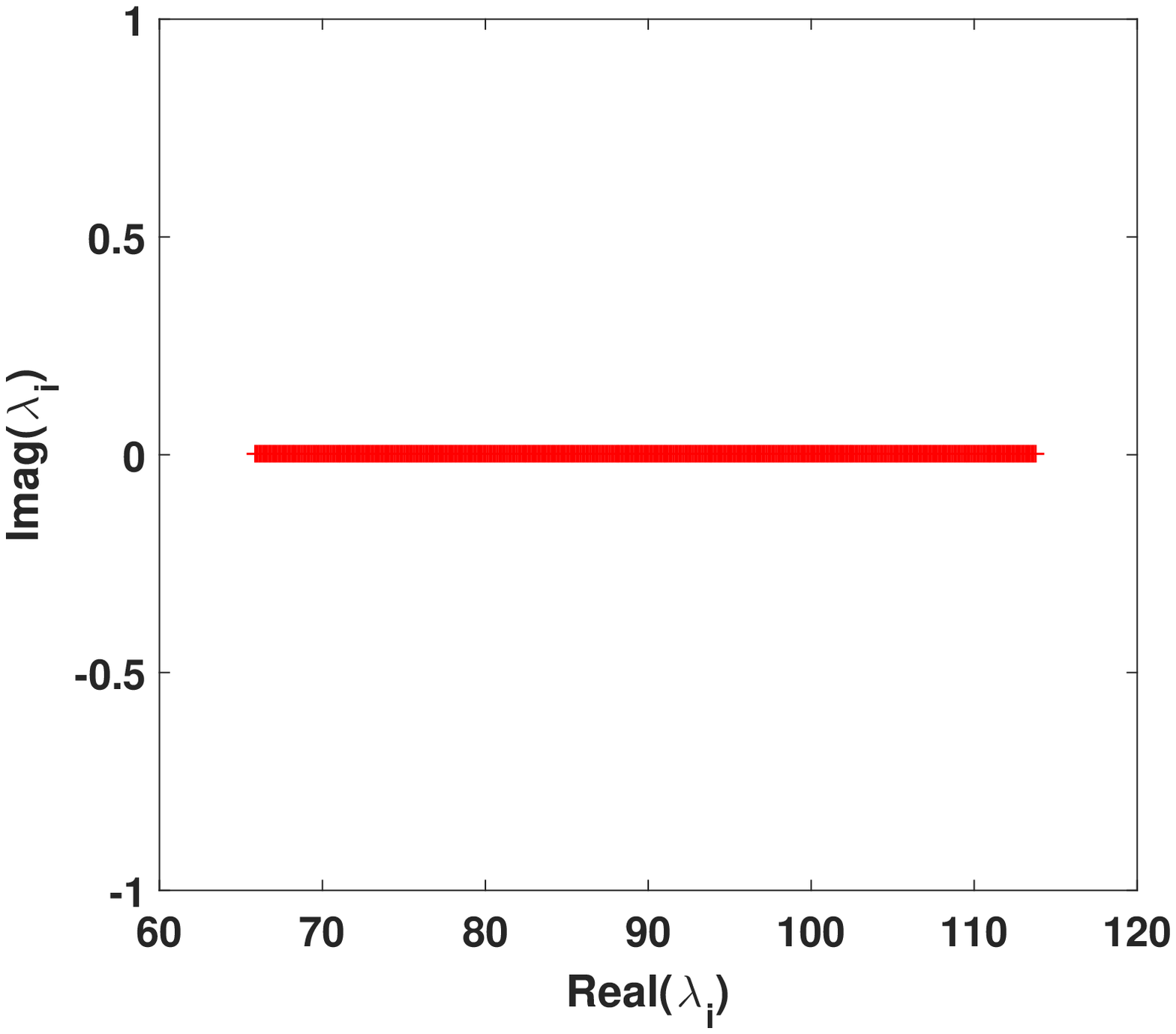}}
     \hspace{1 cm}
  \subfigure{
    \includegraphics[width = 6.7 cm]{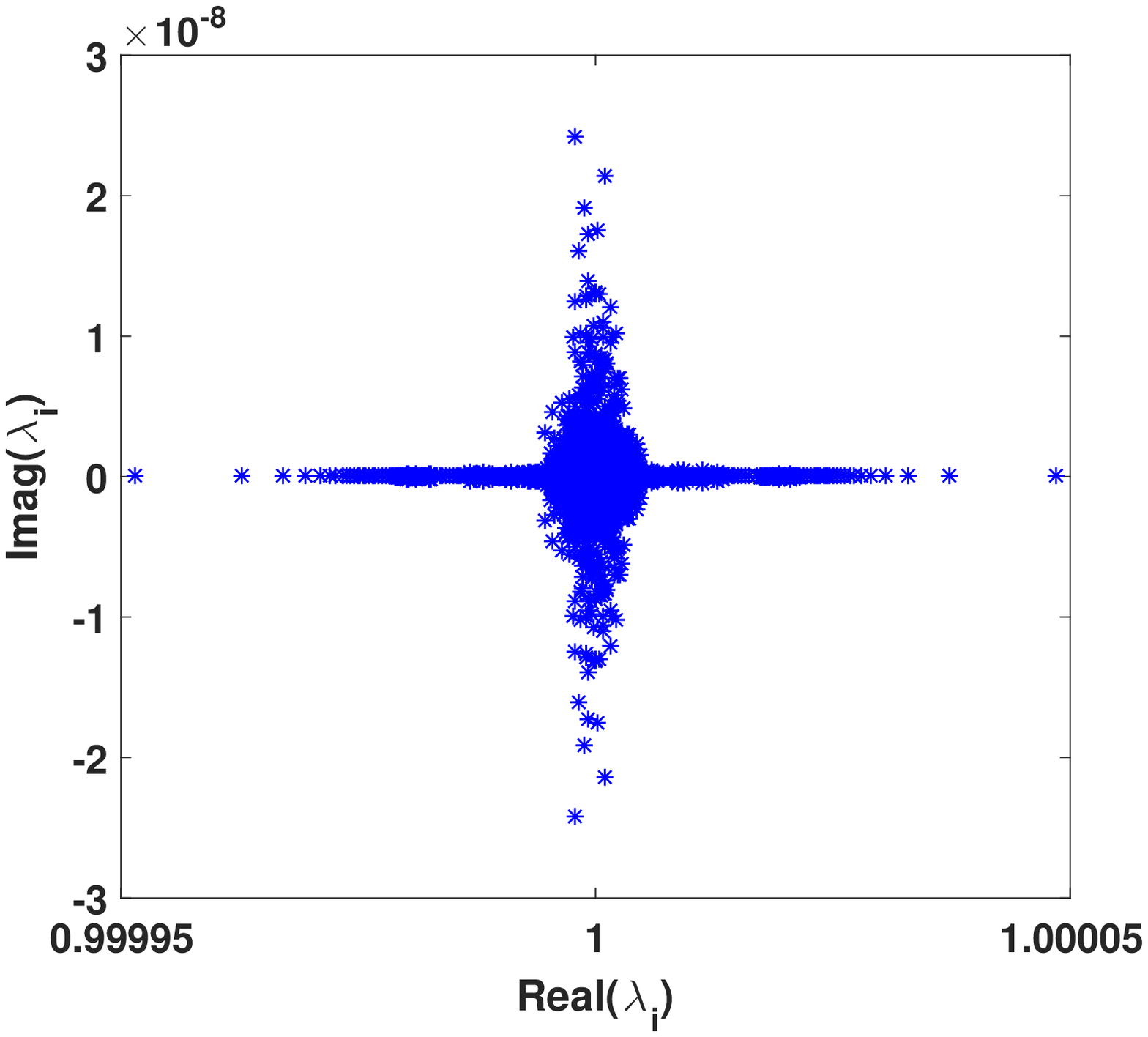}}\\
    (\textbf{a}) Spectrum of $A_2$ \hspace{4.6cm} (\textbf{b}) Spectrum of ${K}_l^{-1}A_2$ \\
  \caption{Spectrum of both original (\textbf{a}) and preconditioned (\textbf{b}) matrices for Example \ref{example2}, where $T$ = 1.5, $J$ = 50, $\beta$ = $\gamma$ = 1.5 and $\tilde{M}$ = $N$ = 128.}
  \label{fig7}
\end{figure}

%table 11
\begin{table}[t]\small\tabcolsep=9.0pt
\begin{center}
\caption{{\small {Comparisons for solving Example \ref{example2} between the Cholesky, PCG, and GL-PCG methods with different $\beta$ and preconditioners, where $l=\frac{M}{2}$, $J$ = 50 and $T$ = 10.}}}
\begin{tabular}{cccccccccc}
\\
\hline & & &\multicolumn{1}{c}{$\rm{Chol}$} & \multicolumn{2}
{c}{$\rm{PCG(R_2)}$}& \multicolumn{2}{c}{GL-PCG($R_2$)} & \multicolumn{2}
{c}{GL-PCG($\mathcal{\tilde{L}}_l$)} \\
[-2pt] \cmidrule(lr){4-4} \cmidrule(lr){5-6} \cmidrule(lr){7-8} \cmidrule(lr){9-10} \\ [-11pt]
 $\beta$=$\gamma$  & $\widetilde{M}$ & $N$ & $\rm{CPU}$ & $\rm{CPU}$ & $\rm{Iter}$ & $\rm{CPU}$ & $\rm{Iter}$ & $\rm{CPU}$ & $\rm{Iter}$ \\
\hline
    &$2^3$ &$2^3$  &0.00  &0.02 &7.0  &0.01 &7.0  &\textbf{0.01} &\textbf{5.0} \\
    &$2^4$ &$2^4$  &0.03  &0.04 &8.0  &0.04 &8.0  &\textbf{0.04} &\textbf{5.0} \\
1.2 &$2^5$ &$2^5$  &0.20  &0.25 &8.0  &0.32 &8.0  &\textbf{0.24} &\textbf{4.0} \\
    &$2^6$ &$2^6$  &2.92  &1.39 &8.0  &1.63 &8.0  &\textbf{1.56} &\textbf{4.0} \\
    &$2^7$ &$2^7$  &65.62 &11.30&7.0  &11.28&7.0  &\textbf{10.39}&\textbf{3.0} \\
  \\
    &$2^3$ &$2^3$  &0.00  &0.01 &7.0  &0.01 &7.0  &\textbf{0.01} &\textbf{5.0} \\
    &$2^4$ &$2^4$  &0.03  &0.04 &9.0  &0.04 &9.0  &\textbf{0.03} &\textbf{4.0} \\
1.5 &$2^5$ &$2^5$  &0.20  &0.28 &9.9  &0.37 &9.9  &\textbf{0.27} &\textbf{4.0} \\
    &$2^6$ &$2^6$  &2.88  &1.42 &9.0  &1.65 &9.0  &\textbf{1.58} &\textbf{4.0} \\
    &$2^7$ &$2^7$  &65.09 &12.39&9.0  &12.37&9.0  &\textbf{10.75}&\textbf{3.0} \\
 \\
    &$2^3$ &$2^3$  &0.00  &0.01 &7.0  &0.01 &7.0  &\textbf{0.01} &\textbf{4.0}  \\
    &$2^4$ &$2^4$  &0.03  &0.04 &9.0  &0.04 &9.0  &\textbf{0.03} &\textbf{4.0} \\
1.9 &$2^5$ &$2^5$  &0.20  &0.29 &11.0 &0.38 &11.0 &\textbf{0.24} &\textbf{3.0} \\
    &$2^6$ &$2^6$  &2.89  &1.55 &12.0 &1.83 &12.0 &\textbf{1.51} &\textbf{3.0} \\
    &$2^7$ &$2^7$  &65.13 &14.01&12.0 &13.88&12.0 &\textbf{10.73}&\textbf{3.0} \\
\hline
\end{tabular}
\label{tab11}
\end{center}
\end{table}

To verify the convergence order of the proposed difference scheme \eqref{6.11}-\eqref{6.13}, in Tables \ref{tab7}-\ref{tab9}, take $h_1=h_2=\widetilde{h}$, and $M_1=M_2=\widetilde{M}$.
Let
$$e(\widetilde{h},\tau,\Delta\alpha)=\max\limits_{{0\leq i\leq M_1},~{0\leq j\leq M_2}\atop{0\leq n\leq N}}
|u(x_i,y_j,t_n, \Delta\alpha)-u_{ij}^n|,$$
where $u(x_i,y_j,t_n,\Delta\alpha)$ and $u_{ij}^n$ represent the exact and numerical solutions at step sizes $\widetilde{h}$, $\tau$ and $\Delta\alpha$, respectively. The convergence orders are defined as
$$\widetilde{rate}_h=\log_2\frac{e(\widetilde{h},\tau,\Delta\alpha)}{e(\widetilde{h}/2,\tau,\Delta\alpha)},~
\widetilde{rate}_{\tau}=\log_2\frac{e(\widetilde{h},\tau,\Delta\alpha)}{e(\widetilde{h},\tau/2,\Delta\alpha)},~
\widetilde{rate}_{\Delta\alpha}=\log_2\frac{e(\widetilde{h},\tau,\Delta\alpha)}{e(\widetilde{h},\tau,{\Delta\alpha}/2)}.$$

Tables \ref{tab10}-\ref{tab12} are to verify the efficiency of the proposed GL-PCG method with the  truncated preconditioner described in Section \ref{section6.2}. The resultant Sylvester matrix equations \eqref{6.16} of the 2D case are solved by using the Cholesky method, the PCG method, and the proposed GL-PCG method, respectively.
The stopping criterions of the PCG and GL-PCG algorithms are
$${\|r^{(k)}\|_2}/{\|r^{(0)}\|_2} < 10^{-9}~\text{and}~
 {\|R_k\|_F}/{\|R_0\|_F} < 10^{-9},$$
%$$
respectively, where $R_k$ is the $k$-th residuals of the GL-PCG algorithm, and the initial guess in each time level is given as the zero matrix.

Besides the proposed truncated preconditioner $\mathcal{\tilde{L}}_l$, we also test the R. Chan's-based BCCB preconditioner:
\begin{equation*}
 R_2=\mu_0I_1+I_2\otimes K_1\nu_{\beta}r(G_\beta)+K_2\nu_{\gamma}r(G_\gamma)\otimes I_3.
\end{equation*}

%table 12
\begin{table}[t]\small\tabcolsep=9.0pt
\begin{center}
\caption{{\small {Comparisons for solving Example \ref{example2} between the Cholesky, PCG, and GL-PCG methods with different $\beta$ and preconditioners, where $l=\frac{M}{2}$, $J$ = 300 and $T$ = 1.5.}}}
\begin{tabular}{cccccccccc}
\\
\hline & & &\multicolumn{1}{c}{$\rm{Chol}$} & \multicolumn{2}
{c}{$\rm{PCG(R_2)}$}& \multicolumn{2}{c}{GL-PCG($R_2$)} & \multicolumn{2}
{c}{GL-PCG($\mathcal{\tilde{L}}_l$)} \\
[-2pt] \cmidrule(lr){4-4} \cmidrule(lr){5-6} \cmidrule(lr){7-8} \cmidrule(lr){9-10} \\ [-11pt]
 $\beta$=$\gamma$  & $\widetilde{M}$ & $N$ & $\rm{CPU}$ & $\rm{CPU}$ & $\rm{Iter}$ & $\rm{CPU}$ & $\rm{Iter}$ & $\rm{CPU}$ & $\rm{Iter}$ \\
\hline
    &$2^3$ &$2^3$  &0.00  &0.01 &5.8  &0.01 &5.6  &\textbf{0.02} &\textbf{4.0} \\
    &$2^4$ &$2^4$  &0.03  &0.03 &5.0  &0.03 &5.0  &\textbf{0.03} &\textbf{3.0} \\
1.2 &$2^5$ &$2^5$  &0.20  &0.22 &5.0  &0.23 &5.0  &\textbf{0.22} &\textbf{3.0} \\
    &$2^6$ &$2^6$  &2.89  &1.20 &4.0  &1.25 &4.0  &\textbf{1.33} &\textbf{3.0} \\
    &$2^7$ &$2^7$  &65.50 &9.14 &3.0  &8.56 &3.0  &\textbf{9.05} &\textbf{2.0} \\
  \\
    &$2^3$ &$2^3$  &0.00  &0.01 &6.0  &0.01 &6.0  &\textbf{0.01} &\textbf{4.0} \\
    &$2^4$ &$2^4$  &0.03  &0.03 &6.0  &0.03 &6.0  &\textbf{0.03} &\textbf{3.0} \\
1.5 &$2^5$ &$2^5$  &0.20  &0.23 &6.0  &0.23 &5.0  &\textbf{0.22} &\textbf{3.0} \\
    &$2^6$ &$2^6$  &2.90  &1.24 &5.0  &1.31 &5.0  &\textbf{1.33} &\textbf{3.0} \\
    &$2^7$ &$2^7$  &65.60 &9.68 &4.0  &8.99 &4.0  &\textbf{9.05} &\textbf{2.0} \\
 \\
    &$2^3$ &$2^3$  &0.00  &0.01 &7.0  &0.01 &7.0  &\textbf{0.01} &\textbf{3.0}  \\
    &$2^4$ &$2^4$  &0.03  &0.03 &7.0  &0.03 &7.0  &\textbf{0.03} &\textbf{3.0} \\
1.9 &$2^5$ &$2^5$  &0.20  &0.24 &7.0  &0.26 &7.0  &\textbf{0.21} &\textbf{3.0} \\
    &$2^6$ &$2^6$  &2.88  &1.29 &6.0  &1.35 &6.0  &\textbf{1.23} &\textbf{2.0} \\
    &$2^7$ &$2^7$  &65.56 &10.56&5.6  &9.74 &5.6  &\textbf{9.03} &\textbf{2.0} \\
\hline
\end{tabular}
\label{tab12}
\end{center}
\end{table}

Figs. \ref{fig4}-\ref{fig5} present the comparisons of the exact and numerical solutions of the difference scheme \eqref{6.11}-\eqref{6.13} for Example \ref{example2} with different $\beta$, $\gamma$ and $T$. We can see that the numerical solutions are in good agreement with the exact solutions.

 The maximum errors and convergence orders of the numerical scheme \eqref{6.11}-\eqref{6.13} for Example \ref{example2} in space, time, and distributed order are displayed in Tables \ref{tab7}-\ref{tab9}, respectively. One can be seen from these tables is that the convergence orders of the scheme \eqref{6.11}-\eqref{6.13} are two. The numerical convergence orders are in accordance with the expected ones.

Figs. \ref{fig6}-\ref{fig7} show the distributions of the eigenvalues of both the original matrix $A_2$ and the proposed truncated preconditioned matrix ${K}_l^{-1}A_2$
($l=\frac{M}{2}$) for Example \ref{example2}. Clearly, the eigenvalues of ${K}_l^{-1}A_2$ are well grouped around 1 and separated away from 0, which confirms that the proposed truncated preconditioner exhibits very nice clustering properties.

Tables \ref{tab10}-\ref{tab12} report the numerical results for Example \ref{example2} by the Cholesky, PCG and GL-PCG methods with preconditioners $\mathcal{\tilde{L}}_l$ and $R_2$ with $l=\frac{M}{2}$.
We can see that the GL-PCG method with preconditioner $\mathcal{\tilde{L}}_l$ exhibits excellent performance. Specifically, the average number of iterations of the GL-PCG($\mathcal{\tilde{L}}_l$) method is much smaller than that of the PCG($R_2$) and GL-PCG($R_2$) methods, and it will not increase with the refinement of the spatial grids.
It also shows that the performances of the PCG($R_2$) and GL-PCG($R_2$) methods are almost the same in terms of the CPU time and the average number of iterations.
The performance of the proposed truncated preconditioner is better than the BCCB preconditioner.

\section{Conclusion}\label{section8}
In this paper, we propose efficient difference methods to solve the time distributed-order and Riesz space fractional diffusion-wave equations with initial-boundary value condition. The unconditional stability and second-order convergence in time, space, and distributed-order of the difference schemes are analyzed. The 1D discretizations lead to SPD Toeplitz linear systems, which are solved by the PCG-based GSF method with R. Chan's circulant preconditioner.
In the 2D case, the GL-PCG method with a truncated preconditioner is designed to solve the discretized SPD Sylvester matrix equations.
Then we show that the convergences of the proposed iterative algorithms are very fast by proving the spectrums of the preconditioned matrices are clustered around one.
%Hence the computational complexity in each iterative step of the proposed method is only $\mathcal{O}(M\log M)$ via using the FFTs.
Numerical experiments are carried out to demonstrate the effectiveness of the proposed numerical methods.
In future work, We will work on developing other fast iterative algorithms and new efficient preconditioners to accelerate the convergence of numerical methods.

\section*{Acknowledgments}
This work is supported by NSFC (61772003 and 11801463), the Applied Basic Research Project of Sichuan Province (20YYJC3482), and the Fundamental Research Funds for the Central Universities (JBK1902028).

%\bibliography{Ref3}

\end{document}